\documentclass[reqno,a4paper]{amsart}
\usepackage[english,activeacute]{babel}
\usepackage[foot]{amsaddr}
\usepackage{amssymb,amsmath,amsthm,amsfonts,mathrsfs,scalefnt,graphicx,graphics,color,hyperref}
\usepackage[font=footnotesize]{caption}
\usepackage[margin=1.05in]{geometry}
\usepackage{lmodern}
\usepackage{textcomp}
\usepackage[on]{auto-pst-pdf}
\usepackage{pst-tree}
\usepackage{pst-all}
\usepackage{tabularx}
\usepackage[T1]{fontenc}
\usepackage{url}  
\usepackage[latin1]{inputenc}
\usepackage{etoolbox}
\usepackage{xcolor,cite}
\usepackage{soul}

\makeatother

\patchcmd{\section}{\scshape}{\bfseries}{}{}
\makeatletter
\renewcommand{\@secnumfont}{\bfseries}
\makeatother
\theoremstyle{plain}
\newtheorem{theorem}{Theorem}[section]
\newtheorem{lemma}[theorem]{Lemma}
\newtheorem{proposition}[theorem]{Proposition}
\newtheorem{corollary}[theorem]{Corollary}

\newtheorem{remark}{Remark}
\numberwithin{equation}{section}
\numberwithin{remark}{section}

\def\Xint#1{\mathchoice
   {\XXint\displaystyle\textstyle{#1}}%
   {\XXint\textstyle\scriptstyle{#1}}%
   {\XXint\scriptstyle\scriptscriptstyle{#1}}%
   {\XXint\scriptscriptstyle\scriptscriptstyle{#1}}%
   \!\int}
\def\XXint#1#2#3{{\setbox0=\hbox{$#1{#2#3}{\int}$}
     \vcenter{\hbox{$#2#3$}}\kern-.5\wd0}}
\def\vp{\Xint=}
\def\vp{\Xint-}

\def\Xint#1{\mathchoice
   {\XXint\displaystyle\textstyle{#1}}%
   {\XXint\textstyle\scriptstyle{#1}}%
   {\XXint\scriptstyle\scriptscriptstyle{#1}}%
   {\XXint\scriptstyle\scriptscriptstyle\scriptscriptstyle{#1}}%
   \!\int}
\def\XXint#1#2#3{{\setbox0=\hbox{$#1{#2#3}{\int}$}
     \vcenter{\hbox{$#2#3$}}\kern-.5\wd0}}

\newmuskip\pFqmuskip
\newcommand*\pFq[6][8]{%
\begingroup%
\pFqmuskip=#1mu\relax%
\mathcode`\,=\string"8000
\begingroup\lccode`\~=`\,
\lowercase{\endgroup\let~}\pFqcomma
{}_{#2}F_{#3}{\left[\genfrac..{0pt}{}{#4}{#5}\Big\arrowvert#6\right]}
\endgroup
}

\newcommand{\pFqcomma}{\mskip\pFqmuskip}
\newcommand{\hf}  {{{}_2F_1}}
\newcommand{\chf}  {{{}_1F_1}}
\newcommand{\HF}  {{{}_3F_2}}
\newcommand{\Cb}  {{\mathbb C}}

\newcommand{\Nb}  {{\mathbb N}}

\newcommand{\Rb}  {{\mathbb R}}

\newcommand{\Zb}  {{\mathbb Z}}
\newcommand{\As} {{\mathcal A}}

\newcommand{\Cs} {{\mathcal C}}

\newcommand{\Ms} {{\mathcal M}}

\newcommand{\Ss} {{\mathcal S}}

\title[]{}
\begin{document}

\section*{ON THE STATIONARY DISTRIBUTION OF THE BLOCK COUNTING PROCESS FOR POPULATION MODELS WITH MUTATION AND SELECTION}
   {\sc F. Cordero}\footnote{Faculty of Technology, Universit\"{a}t Bielefeld, Universit\"{a}tsstrasse 25, 33615 Bielefeld, Germany. E-mail: fcordero@techfak.uni-bielefeld.de}, {\sc M. M\"ohle}\footnote{Mathematisches Institut, Eberhard Karls Universit\"at T\"ubingen,
   Auf der Morgenstelle 10, 72076 T\"ubingen, Germany, E-mail address: martin.moehle@uni-tuebingen.de}
   \begin{center}
 \today
\end{center}
\begin{abstract}
We consider two population models subject to the evolutionary forces of selection and mutation, the Moran model and the $\Lambda$-Wright--Fisher model. In such models the block counting process traces back the number of potential ancestors of a sample of the population at present. Under some conditions the block counting process is positive recurrent and its stationary distribution is described via a linear system of equations. In this work, we first characterise the measures $\Lambda$ leading to a geometric stationary distribution, the Bolthausen--Sznitman model being the most prominent example having this feature. Next, we solve the linear system of equations corresponding to the Moran model. For the $\Lambda$-Wright--Fisher model we show that the probability generating function associated to the stationary distribution of the block counting process satisfies an integro differential equation. We solve the latter for the Kingman model and the star-shaped model. 
 \vspace{2mm}

   \noindent Keywords: Kingman coalescent; star-shaped coalescent; Bolthausen--Sznitman coalescent; Wright--Fisher model; Moran model.
\vspace{2mm}

   \noindent 2010 Mathematics Subject Classification:
            Primary 45J05; 60J27; 97K60 Secondary 44A60; 92D15
\end{abstract}
\maketitle
\section{Introduction}\label{s1}
There is a large variety of population models describing the interplay between mutation and selection forwards in time. Understanding the underlying ancestral processes  
is a major challenge in population genetics. In neutral population models, ancestries are typically described by coalescent processes. The most important example is Kingman's coalescent \cite{Ki82}, which only allows for mergers of pairs of ancestral lineages. Kingman shows in \cite{Ki82} that this process arises as the limit of the genealogies of the neutral Moran and Wright--Fisher models when the population size tends to infinity. Convergence to the Kingman coalescent holds for a wide class of neutral population models (see \cite{MM00}). However, in some situations Kingman's coalescent is not a suitable approximation, which leads to consider coalescent processes that allow for multiple mergers. Exchangeable coalescents with multiple mergers (but without simultaneous multiple mergers) are characterised by a finite measure $\Lambda$ on $[0,1]$, and therefore called $\Lambda$-coalescents. They were introduced in \cite{DK99}, \cite{Pit99} and \cite{Sa99}, and have been the subject of extensive research in the past decades (see, e.g., \cite{Pit06,B09}). The case $\Lambda=\delta_{0}$ corresponds to Kingman's coalescent, the case $\Lambda=\delta_{1}$ to the star-shaped coalescent, and the case where $\Lambda$ is the uniform distribution on $(0,1)$ to the Bolthausen--Sznitman coalescent \cite{BS98}. The $\Lambda$-coalescent specifies the genealogy of a sample from a forwards in time population model introduced in \cite{BLG03} and commonly referred as the $\Lambda$-Fleming--Viot process (see also \cite{BB09} for a review on the topic). Moreover, the block counting process of the $\Lambda$-coalescent is moment dual to the type-frequency process of the $\Lambda$-Fleming--Viot process (see, e.g., \cite{E11}). Formulas for the infinitesimal rates of the block counting process are provided in \cite{He15} for the $\Lambda$-coalescents and in \cite{GM16} for the full class of exchangeable coalescents.

Generalisations of the $\Lambda$-Fleming--Viot processes that incorporate selection and mutation are studied in \cite{EGT10} (the Kingman case is treated in \cite{EG09}). In particular, under mild conditions assuring the existence of a stationary distribution for the type-frequency process, a moment duality between the latter and a process describing the genealogy of a typed sample of the population is established there. In this paper, we focus on $\Lambda$-Fleming--Viot processes with two types, parent independent mutation and genic selection, and we refer to them as $\Lambda$-Wright--Fisher models with mutation and selection. In this context, genealogies can be alternatively described by means of the $\Lambda$-ancestral selection graph ($\Lambda$-ASG); originally constructed in \cite{KroNe97, NeKro97} for the Wright--Fisher diffusion, and extended to the $\Lambda$-Wright--Fisher model in \cite{BLW16} (see also \cite{GS18} for $\Lambda$-Wright--Fisher models with frequency dependent selection). In the $\Lambda$-ASG the coalescence mechanism is given by the $\Lambda$-coalescent. Additionally, selection introduces binary branching at constant rate per ancestral line. This approach differs from the one in \cite{EGT10} in that the ASG describes the ancestries of an untyped sample of the population. In absence of mutations, the block counting process of the $\Lambda$-ASG is in moment duality with the type-frequency process in the $\Lambda$-Wright--Fisher model (see, e.g., \cite{GS18}). In presence of mutations, two relatives of the $\Lambda$-ASG permit to resolve mutation events on the spot and encode relevant information of the model: the $\Lambda$-killed ASG and the $\Lambda$-pruned lookdown ASG (the three ancestral processes coincide in absence of mutations). The killed ASG is reminiscent to the coalescent with killing \cite[Chap.~1.3.1]{D08} and it was introduced in \cite{BW18} for the Wright--Fisher diffusion model with selection and mutation. Its construction generalises in a natural way to $\Lambda$-Wright--Fisher models. The killed ASG helps to determine weather or not all the individuals in a sample of the population at present are unfit and is related to the type-frequency process via a moment duality. The pruned lookdown ASG in turn helps to determine the type of the common ancestor of a given sample of the population. It was introduced in \cite{LKBW15} for the Wright--Fisher diffusion model and extended to the $\Lambda$-Wright--Fisher model in \cite{BLW16} and to the Moran model and its deterministic limit in \cite{FC17} (see also \cite{BCH17}). Moreover, the block counting process of the pruned lookdown ASG is Siegmund dual to the fixation line process (see \cite{He15,GM16} for the neutral case and \cite{BLW16} for the general case). In what follows, if not explicitly mentioned, whenever we talk of the block counting process we refer to the block counting process of the $\Lambda$-pruned lookdown ASG.

In the Wright--Fisher diffusion and the Moran model the block counting process is positive recurrent for any strictly positive selection parameter. For the $\Lambda$-Wright--Fisher model, there is a critical value $\sigma_\Lambda$ such that the block counting process is positive recurrent for any selection parameter $\sigma\in(0,\sigma_\Lambda)$ (see \cite{Fo13} and the discussion in \cite[p.~4]{BLW16}). The constant $\sigma_\Lambda$ was introduced in \cite{HeMo12} as $\lim_{k\to\infty}\log{k}/E_k[T_1]$, where $T_1$ is the absorption time of the block counting process of the $\Lambda$-coalescent. Moreover, in absence of mutations, fixation of the fit type is certain if and only if $\sigma\geq \sigma_\Lambda$ (see \cite{DEP11,DEP12} for the Eldon--Wakeley coalescent, \cite{Fo13,G14} for two independent proofs in the general  case and \cite{GS18} for models with frequency dependent selection). In this paper, we are interested in the stationary tail probabilities of the block counting process. In the Wright--Fisher diffusion model they are characterised via a two-step recurrence relation (see \cite{LKBW15}), which is referred in the literature as Fearnhead's recursion (see also \cite{Fe02,Ta07}). Linear systems of equations that characterise the stationary tail probabilities of the block counting process are provided in \cite{BLW16} for the $\Lambda$-Wright--Fisher model and in \cite{FC17} for the Moran model. We refer to these linear systems as Fearnhead-type recursions. The stationary moments of the type-frequency process are characterised via similar linear systems (see \cite{EGT10} and Corollary \eqref{morecur}).

On the basis of the Fearnhead(-type) recursions, we aim to identify the measures $\Lambda$ such that the stationary distribution of the block counting process is geometrically distributed, the measure $\Lambda\equiv 0$ being known to have this feature \cite{BCH17}. Next, we aim to provide explicit expressions for the stationary distribution of the block counting process for the Moran model and some particular cases of the $\Lambda$-model, namely, the Kingman model, the star-shaped model and the Bolthausen--Sznitman model. For the general $\Lambda$-model we will characterise the probability generating function of the stationary distribution of the block counting process via an integro differential equation. 

The paper is organised as follows. In Section \ref{s2} we briefly describe the Moran model and the $\Lambda$-Wright--Fisher model with selection and mutation together with their corresponding block counting process. For both models we recall the characterisation of the stationary tail probabilities of the block counting process via the Fearnhead-type recursions. A similar system of equations for the moments of the asymptotic proportion of unfit individuals is provided via a moment duality. In Section \ref{s3} we characterise the measures $\Lambda$ leading to a geometric distribution and we provide a class of measures having this feature. The Bolthausen--Sznitman coalescent is the most prominent example (see Corollary \ref{geomlawBS}) and it is the only $\beta(a,b)$-model having a geometric stationary distribution (see Remark \ref{remgeo}). In Section \ref{s4} we treat the Moran model with mutation and selection. We obtain formulas for the probability mass function, the probability generating function, the mean and the factorial moments of the stationary distribution of the block counting process. In Section \ref{s5} we characterise the probability generating function of the stationary distribution of the block counting process for the $\Lambda$-Wright--Fisher model by an integro differential equation. In Section \ref{s6} we obtain formulas for the probability mass function, the probability generating function, the mean and the factorial, moments of the stationary distribution of the block counting process in the Wright--Fisher model. Section \ref{s7} treats the star-shaped model with mutation and selection. In Section \ref{s8} we come back to the Bolthausen--Sznitman model. We first relate the geometric law with the results obtained in Section \ref{s5} for the general $\Lambda$-model. In addition, we obtain an explicit expression for the generating function of the moments of the asymptotic frequency of unfit individuals. In Section \ref{s9} we comment on the $\beta(3,1)$-model, which provides another example for which the Fearnhead-type recursions can be solved explicitly.

\section{Preliminaries: Fearnhead-type recursions}\label{s2}
\subsection{The block-counting process of the pruned lookdown ASG}
In the two-types Moran model of size $N>1$ each individual is characterised by a type $i\in\{0,1\}$. If an individual reproduces, its single offspring  inherits the parent's type and replaces a uniformly chosen individual, possibly its own parent. The replaced individual dies, keeping the size of the population constant. Individuals of type $1$ reproduce at rate $1$, whereas individuals of type $0$ reproduce at rate $1+s$, $s>0$. Mutation occurs independently of reproduction. Moreover, each individual mutates to type $j\in\{0,1\}$ at rate $u_j\geq 0$. Hence, the total rate of mutation per individual is $u:=u_0+u_1$. 
Relevant information of the model is given by the continuous-time Markov chain $X^N:=(X^N_t)_{t\geq 0}$ describing the evolution in time of the number of type-$0$ individuals in the population. Its generator is given by
$$\As_{X^N}f(k):=\lambda_k^N (f(k+1)-f(k))+\mu_k^N (f(k-1)-f(k)),\quad k\in [N]_0:=\{0,\ldots,N\},$$
where $\lambda_k^N:=k(N-k)(1+s)/N+(N-k)u_0$ and $\mu_k^N:=k(N-k)/N+ ku_1$. The asymptotic properties of $X^N$ are well known: (1) if $u_0,u_1>0$, $X^N$ admits a unique stationary distribution, which is given by $\pi_N(k):=C_N\prod_{i=1}^k \lambda_{i-1}^N/\mu_i^N$, $k\in [N]_0$, where $C_N$ is a normalising constant, (2) if $u_0=0$ and $u_1>0$, $X^N$ fixates almost surely at $0$, (3) if $u_0>0$ and $u_1=0$, $X^N$ fixates almost surely at $N$, (4) if $u_0=u_1=0$, conditionally on $X_0^N=k$, $X^N$ fixates at $N$ with probability $((1+s)^{N}-(1+s)^{N-k})/((1+s)^N-1)$  (c.f. \cite[Chap.~6.1.1]{D08}).

Backward in time potential ancestors of a sample of the population are traced back with the help of the ASG. An appropriate dynamical ordering and pruning of its lines leads to the pruned lookdown ASG, which in turn permits to characterise the common ancestor type distribution (see \cite{FC17}). The block counting process $L^N:=(L^N_t)_{t\geq 0}$ of the pruned lookdown ASG describes the number of potential ancestors of a given sample of individuals. It is a continuous time Markov chain with state space $[N]:=\{1,\ldots,N\}$ and infinitesimal rates
\begin{equation}\label{rmm}
q_{N}(i,j):=\frac{i\,(N-i)\,s}{N}\, 1_{\{j=i+1\}}+\left( \frac{i(i-1)}{N} +(i-1) \,u_1+u_0\right)1_{\{j=i-1\}} +u_0\,1_{\{j\leq i-2\}},\quad i,j\in[N].
\end{equation}
The process $L^N$ is irreducible, and has hence a unique stationary distribution $(p_n^N)_{n\in[N]}$. Let $L_\infty^N$ be a random variable distributed according to $(p_n^N)_{n\in[N]}$. The stationary tail probabilities $a_n^N:=P(L_\infty^N>n)$, $n\in[N-1]_0$, are characterised by the recurrence relation (see \cite[Prop.~4.7]{FC17})
\begin{equation}\label{frmm}
\left(\frac{n}{N}+u_1\right)a_n^N=\left(\frac{n}{N}+\frac{N-n+1}{N}\,s+u\right)a_{n-1}^N-\frac{N-n+1}{N}\,s\,a_{n-2}^N,\quad n\in\{2,\ldots,N-1\},
\end{equation}
together with the boundary conditions
\begin{equation}\label{bcmm}
 a_0^N=1\quad\textrm{and}\quad\left(1+u+\frac{s}{N}\right)a_{N-1}^N=\frac{s}{N}a_{N-2}^N.
\end{equation} 
Depending on the strengths of selection and mutation two standard limits of large populations arise in the Moran model. The first one assumes that the parameters of selection and mutation remain constant with respect to the size of the population (strong selection - strong mutation). In this case, a special case of the dynamical law of large numbers of Kurtz \cite[Thm.~3.1]{K70} permits to show that the proportion of type-$0$ individuals converges to the solution of the haploid mutation-selection equation (see \cite{FC17b} for more details)
 \begin{equation}
z'(t)=sz(t)(1-z(t))+u_0(1-z(t))-u_1 z(t), \quad t\geq 0,\label{deteq}
\end{equation} 
This is a classical model in population genetics and goes back to Crow and Kimura \cite{CK56}. However, the (random) ancestral structures inherent to this model have been only recently investigated (see \cite{FC17, BCH17}). Note that the same limit is obtained in a moderate selection - moderate mutation regime of the Moran model, i.e. if $s$, $u_0$ and $u_1$ are of order $N^{-\alpha}$ for some $\alpha\in(0,1)$ and time is rescaled by a factor $N^{\alpha}$. This follows by a straightforward Taylor expansion of the generator of the rescaled process and standard convergence results for continuous time Markov chains (e.g. \cite[Thms.~1.6.1,~4.2.11~and~8.2.1]{EK86}).

The other asymptotic regime arises when $s\sim \sigma/N$, $u_0\sim \theta_0/N$ and $u_1\sim \theta_1/N$, for some $\sigma,\theta_1,\theta_0\geq 0$ (weak selection - weak mutation). In this case, rescaling time by $N$, the proportion of fit individuals converges to the Wright--Fisher diffusion process with infinitesimal generator
$$\As_Y f(x):=x(1-x)f{''}(x)+ \left[\sigma x(1-x)+\theta_0(1-x)-\theta_1x\right]f{'}(x),\quad f\in\Cs^2([0,1]),\,x\in[0,1].$$
These two infinite population models are particular cases of the $\Lambda$-Wright--Fisher model. 
The $\Lambda$-Wright--Fisher model describes a two-types infinite population evolving according to random reproduction, two-way mutation and fecundity selection. The parameters of the model are (1) a finite measure $\Lambda$ on $[0,1]$ modelling the neutral reproduction, (2) the selective advantage $\sigma\in\Rb_+:=[0,\infty)$ and (3) the mutation rates $\theta_0,\theta_1\in \Rb_+$. The process $X$ describing the frequency of type $0$ in the population has the generator
\begin{align*}
A_{X} f(x)&:=\int_{(0,1]}\frac{\Lambda({\rm d}z)}{z^2}\left[x(f(x+z(1-x))-f(x))+(1-x)(f(x-zx)-f(x))\right]\\
&+\frac{\Lambda(\{0\})}{2}x(1-x)\,f''(x)+\left[\sigma x(1-x)+\theta_0(1-x)-\theta_1x\right]\,f'(x),\quad f\in\Cs^2([0,1]),\,x\in[0,1].
\end{align*}
\begin{remark}
 The case $\Lambda=2\delta_{0}$, where $\delta_0$ is the Dirac mass at $0$, corresponds to the Wright--Fisher diffusion model.
\end{remark}
\begin{remark}[Crow--Kimura and seed bank models]
 The degenerate case $\Lambda\equiv 0$, meaning that there is no neutral reproduction, corresponds to the Crow--Kimura model, i.e. the solution of the ODE \eqref{deteq} with $s=\sigma$, $u_0=\theta_0$ and $u_1=\theta_1$. One may think of a population of seeds which do not reproduce, but forces like mutation and (viability) selection may still act on the seeds since they are exposed to heat, chemicals or radiation. The study of seed banks models is an active area of research and goes back to the seminal paper of Kaj. et al. \cite{KKL01}. In the case $u=\theta=0$, the ODE \eqref{deteq} can be recovered from the seed bank model with geometric germination rate and weak selection \cite[Eq.~(5)]{KMTZ17} by assuming that the selection intensity is proportional to the germination rate and letting the latter go to $0$. In the case $s=\sigma=0$,  the ODE \eqref{deteq} corresponds to the seed bank component in \cite[Eq.~(1)]{BBGW18} when the relative seed bank size $K$ tends to zero.
\end{remark}

In \cite{BLW16} the $\Lambda$-pruned lookdown ASG was defined in order to trace back ancestries in the $\Lambda$-Wright--Fisher model. The way ancestral lines are pruned in this process is tailored to compute the common ancestor type distribution (see Remark \ref{catd}). The corresponding block counting process $L^\Lambda:=(L^\Lambda_t)_{t\geq 0}$ is the continuous time Markov chain with state space $\Nb:=\{1,2,\ldots\}$ and infinitesimal generator
\begin{align*}
 G_{L^\Lambda} g(k):&=\sum\limits_{\ell=1}^{k-1}\binom{k}{k-\ell+1}\lambda_{k,k-\ell+1}[g(\ell)-g(k)]+k\sigma[g(k+1)-g(k)]\\
 &+(k-1)\theta_1[g(k-1)-g(k)]+\sum\limits_{\ell=1}^{k-1}\theta_0[g(k-\ell)-g(k)],\quad g:\Nb\to\Rb,\, k\in\Nb,
\end{align*}
where $\lambda_{k,j}:=\int_{[0,1]} x^j(1-x)^{k-j}x^{-2}\Lambda({\rm d}x)$, $2\leq j\leq k$. 

Let $\sigma_\Lambda:=-\int_{[0,1]} \log(1-x)\,\frac{\Lambda({\rm d}x)}{x^2}$. In \cite{Fo13} it is shown that if $\sigma\in(0,\sigma_\Lambda)$ and $\theta=0$, then the process $L^\Lambda$ is positive recurrent. The next result improves this condition for $\theta:=\theta_0+\theta_1>0$.
\begin{lemma}\label{posrec}
Assume that $\sigma>0$. If $\theta_0>0$ or $\sigma<\sigma_\Lambda+\theta_1$, then the process $L^\Lambda$ is positive recurrent.  
\end{lemma}
\begin{proof}
 If $\sigma_\Lambda=\infty$ the result is already covered in \cite{Fo13}. In the case $\sigma_\Lambda<\infty$, we follow the proof of \cite[Lemma~2.4]{Fo13}. Let $T_k:=\inf\{s\geq 0: L_s^\Lambda<k\}$, $k\in\Nb$. We will show that there exists $n_0\in\Nb$, such that for all $n\geq n_0$, $E_n[T_{n_0}]<\infty$. If $\theta_0>0$ and $L_0^\Lambda=n\geq 2$, $T_2$ is dominated by an exponential time with parameter $\theta_0$, and the result follows in this case. Now we assume that $\theta_0=0$ and that $\sigma-\theta_1<\sigma_\Lambda$. If $\theta_1>\sigma$,  $L^\Lambda$ is dominated by a birth and death process with birth rate $\sigma$ and death rate $\theta_1$, which is positive recurrent. Hence $L^\Lambda$ is positive recurrent. At last we consider the case where $\sigma\geq \theta_1$. We define for $n\geq 2$ and $\ell\in\Nb$
 $$\delta(n):=-n\int_{[0,1]}\log\left(1-\frac{1}{n}(nx-1+(1-x)^n)\right)\frac{\Lambda({\rm d}x)}{x^2}\quad\textrm{and}\quad f(\ell):=\sum_{k=2}^\ell \frac{k}{\delta(k)}\log\left(\frac{k}{k-1}\right).$$
 A slight modification of the proof of \cite[Lemma~2.3]{Fo13} permits to show that
 $$G_{L^\Lambda} f(\ell)\leq -1+(\sigma-\theta_1)\frac{\ell}{\delta(\ell)},\quad \ell\geq 2.$$
Set $f_N(\ell):=f(\ell)1_{\{\ell\leq N+1\}}$, $N\in\Nb$. By Dynkin's formula the process $(f_N(L_t^\Lambda)-\int_0^t G_{L^\Lambda} f_N(L_s^\Lambda){\rm d}s)_{t\geq 0}$ is a martingale. Since $\lim_{n\to\infty}\delta(n)/n=\sigma_\Lambda$ (see \cite[Remark~4.3]{HeMo12}), we infer that for any $\epsilon>0$ there is $n_0\in\Nb$ such that for all $\ell\geq n_0$, $\ell/\delta(\ell)\leq \sigma_\Lambda^{-1}+\epsilon$. Consider the stopping time $S_N:=\inf\{s\geq 0: L_s^\Lambda>N\}$. Applying the optional stopping theorem and using that $G_{L^\Lambda} f_N(\ell)=G_{L^\Lambda} f(\ell)$ for $\ell\leq N$ yields for $n_0\leq n\leq N$
 \begin{align*}
  E_n[f_N(L^\Lambda_{T_{n_0}\wedge S_N \wedge k})]&=f_N(n) +\int_0^{T_{n_0}\wedge S_N \wedge k} G_{L^\Lambda} f_N(L_s^\Lambda) {\rm d}s\\
  &\leq f_N(n) +\int_0^{T_{n_0}\wedge S_N \wedge k} \left(-1+(\sigma-\theta_1)\frac{L_s^\Lambda}{\delta(L_s^\Lambda)}\right) {\rm d}s\\
  &\leq f_N(n) +\left(-1+\frac{(\sigma-\theta_1)}{\sigma_\Lambda}+\epsilon(\sigma-\theta_1)\right)E_n\left[T_{n_0}\wedge S_N \wedge k\right].
 \end{align*}
 Therefore,
 $$\left(1-\frac{\sigma-\theta_1}{\sigma_\Lambda}-\epsilon(\sigma-\theta_1)\right)E_n[T_{n_0}\wedge S_N\wedge k]\leq  f_N(n)- E_n[f_N(L^\Lambda_{T_{n_0}\wedge S_N \wedge k})]\leq f_N(n).$$
We choose $\epsilon>0$ such that $1-(\sigma-\theta_1)(\sigma_\Lambda^{-1}+\epsilon)>0$. Since the process $L^\Lambda$ is non-explosive (it is dominated by a Yule process with parameter $\sigma$), $S_N\rightarrow\infty$ as $N\to\infty$. Letting $N\to\infty$ in the previous inequality yields, for all $n\geq n_0$,
 $$\left(1-\frac{\sigma-\theta_1}{\sigma_\Lambda}-\epsilon(\sigma-\theta_1)\right)E_n[T_{n_0}\wedge k]\leq f(n).$$
 Letting $k\to \infty$ yields $E_n[T_{n_0}]\leq f(n)$. The proof is achieved.
 \end{proof}
\begin{remark}
 From the results in \cite{HeMo12} it follows that if the $\Lambda$-coalescent comes down from infinity, then $\sigma_\Lambda=\infty$. Therefore, for measures $\Lambda$ having this property, the block-counting process is always positive recurrent.
\end{remark}

Under the assumptions of Lemma \ref{posrec}, the process $L^\Lambda$ admits a unique stationary measure $(p_n^\Lambda)_{n\in\Nb}$. We denote by $L_\infty^\Lambda$ a random variable distributed according to $(p_n^\Lambda)_{n\in\Nb}$. From \cite[Thm.~2.4]{BLW16} we know that the sequence of stationary tail probabilities $a_n^\Lambda:=P(L_\infty^\Lambda>n)$, $n\in\Nb_0:=\Nb\cup\{0\}$, is the unique solution of the system of equations
\begin{equation}\label{frlm}
 \frac{1}{n}\sum\limits_{k=2}^\infty\binom{k+n-1}{k}\lambda_{k+n,k}(a_n^\Lambda-a_{k+n-1}^\Lambda)+ \left(\frac{m_1}{n}+\sigma+\theta\right)a_n^\Lambda
 =\sigma a_{n-1}^\Lambda+\theta_1 a_{n+1}^\Lambda,\quad n\in\Nb,
 \end{equation}
with $m_1:=\Lambda(\{1\})$, together with the boundary conditions $a_0^\Lambda=1$ and $\lim_{n\rightarrow\infty} a_{n}^\Lambda=0.$ For $\Lambda=2\delta_{0}$, \eqref{frlm} reduces to Fearnhead's recursion (see \cite{Fe02,Ta07,LKBW15}). For this reason, we refer to \eqref{frmm} and to \eqref{frlm} as Fearnhead-type recursions. 
\begin{remark}\label{catd}
Assume that $L^\Lambda$ is positive recurrent. Denote by $J_t$ the type of the ancestor at time $0$ of an individual randomly sampled from the population at time $t$. Then \cite[Thm.~2.4]{BLW16} also shows that
$$h_\Lambda(x):=\lim_{t\to\infty} P(J_t=0\mid X_0=x)=\sum_{n=0}^\infty x(1-x)^n a_n^\Lambda,\quad x\in[0,1].$$
\end{remark}
\subsection{Moment duality}
Under the assumption that $X$ has a stationary distribution (for example if $\theta_0,\theta_1>0$), the process $(X_t,1-X_t)_{t\geq 0}$ is in moment duality with a branching-coalescing process, following the typed ancestry of a given sample of the population (see \cite{EG09} for the Wright--Fisher diffusion model  and \cite{EGT10} for general $\Lambda$-Fleming--Viot models with seletion and mutation). Moment dualities can be also constructed, in an untyped manner, on the basis of the ASG. In absence of mutations, the block counting process of the ASG, which in this case coincides with $L^\Lambda$, is moment dual to the process $1-X$ (see, e.g., \cite{GS18}). In presence of mutations the moment duality between these two processes does not hold anymore. In \cite{BW18} the notion of killed ASG is introduced for the Wright--Fisher diffusion model and it is shown that this process is moment dual to $1-X$. This construction extends in a natural way to the $\Lambda$-Wright--Fisher model as follows. In addition to the neutral mechanism given by the$\Lambda$-coalescent, and the selective binary branching, every line is pruned at rate $\theta_1$ and the whole process is killed at the first beneficial mutation event. The corresponding block-counting process $(R_t)_{t\geq 0}$ has infinitesimal generator  
\begin{align*}
 G_{R} g(k):&=\sum\limits_{\ell=1}^{k-1}\binom{k}{k-\ell+1}\lambda_{k,k-\ell+1}[g(\ell)-g(k)]+k\sigma[g(k+1)-g(k)]\\
 &+k\theta_1[g(k-1)-g(k)]+k\theta_0 [g(\Delta)-g(k)],\quad g:\Nb\cup\{\Delta\}\to\Rb,\, k\in\Nb,
\end{align*}
where $\Delta$ is a cemetery point. To the best of the authors knowledge the following moment duality for the general two-way mutation selection $\Lambda$-Wright--Fisher model has not been stated in the literature so far. 
\begin{proposition}[Moment duality]\label{mdlwfm} For all $n\in\Nb_0$ and $x_0\in[0,1].$
\begin{equation}\label{eqmd}
 E_{x_0}[(1-X_t)^n]=E_n[(1-x_0)^{R_t}],\qquad t\geq 0
\end{equation}
 with the convention $y^\Delta=0$ for any $y\in[0,1]$.
\end{proposition}
\begin{proof}
Let $H:[0,1]\times \Nb_0\cup\{\Delta\}$ be the function defined via $H(x,n):=(1-x)^n$, $n\in\Nb_0$, $x\in[0,1]$, and $H(x,\Delta)=0$ for all $x\in[0,1]$. Applying the generator $G_R$ to the function $H(x,\cdot)$ and rearranging terms leads to
\begin{align}\label{peq1}
G_RH(x,\cdot)(n)&=\sum\limits_{j=2}^{n}\binom{n}{j}\lambda_{n,j}\left((1-x)^{n+1-j}-(1-x)^n\right)-n(1-x)^{n-1}\left(\sigma x(1-x)-\theta_1 x+\theta_0(1-x)\right).
\end{align}
Moreover, using the definition of the coefficients $\lambda_{n,j}$, $2\leq j\leq n$, applying the binomial theorem and rearranging terms, we obtain
\begin{align}\label{peq2}
& \sum\limits_{j=2}^{n}\binom{n}{j}\lambda_{n,j}\left((1-x)^{n+1-j}-(1-x)^n\right)=\frac{\Lambda(\{0\})}{2}n(n-1)(1-x)^{n-1}x\nonumber\\
 &\qquad+\int_{(0,1]}\frac{\Lambda(dz)}{z^2}\left\{x[((1-x)(1-z))^n-(1-x)^n]+(1-x)[(1-x(1-z))^n-(1-x)^n]\right\}.
\end{align}
Note that $H(\cdot,n)'(x)=-n(1-x)^{n-1}$ and $H(\cdot,n)''(x)=n(n-1)(1-x)^{n-2}$. Thus, plugging \eqref{peq2} in \eqref{peq1} yields $G_R H(x,\cdot)(n)=A_X H(\cdot,n)(x)$, $n\in\Nb_0$ and $x\in[0,1]$. The result follows from \cite[Thm.~3.42]{Li10} (or \cite[Prop.~1.2]{JK14}).
\end{proof}

Let us assume now that $\theta_0$ and $\theta_1$ are strictly positive. In this case, the process $R$ is absorbed either in $0$ or in $\Delta$. Moreover, the process $X$ is positive recurrent and the moments of its stationary distribution are characterised in the following corollary. 
\begin{corollary}[Moments of the stationary distribution]\label{morecur}
Assume that $\theta_0,\theta_1>0$ and let $X_\infty$ be a random variable distributed according to the stationary distribution of the process $X$. Then, for all $n\in\Nb_0$
$$E[(1-X_\infty)^n]=P_n(R_\infty=0)=:w_n.$$
Moreover, the sequence $(w_n)_{n\in\Nb_0}$ is characterised via
\begin{equation}\label{momrec}
 \left(\theta+\sigma+\frac1n\sum_{\ell=1}^{n-1}\binom{n}{n-\ell+1}\lambda_{n, n-\ell+1}\right)w_n=\theta_1 w_{n-1}+\sigma w_{n+1}+\frac1n\sum_{\ell=1}^{n-1}\binom{n}{n-\ell+1}\lambda_{n, n-\ell+1}w_\ell,
\end{equation}
for $n\in\Nb$ and $w_0=1$.
\end{corollary}
\begin{proof}
 The first part follows directly by letting $t\to\infty$ in \eqref{eqmd}. Equation \eqref{morecur} is obtained by decomposing the event $\{R_\infty=0\}$ according to the first step of $R$ away from its initial state.
\end{proof}
\begin{remark}
 In the case $\Lambda\equiv 0$, it is known that $w_n=w^n$, where $w$ is the smallest root in $[0,1]$ of $\sigma w^2-(\theta+\sigma)w+\theta_1=0$ (see \cite{BCH17}). 
\end{remark}
The main focus of this work is on solving the Fearnhead-type recursions \eqref{frmm} and \eqref{frlm}. Nevertheless, we also provide an explicit expression for the generating function of the coefficients $w_n$ in Section \ref{s8} for the Bolthausen--Sznitman model.

In what follows it will be convenient to decompose $\Lambda:= m_0 \delta_{0}+\Lambda_0$, where $\Lambda_0$ has no mass at $0$.

\section{\texorpdfstring{Geometric law in the $\Lambda$-model}{Geometric law in the Lambda-model}}\label{s3}
For the Crow--Kimura model ($\Lambda\equiv 0$), it has been shown in \cite{BCH17} (see also \cite{FC17}) that the block counting process is positive recurrent if and only if $\theta_0>0$ or $\theta_0=0$ and $\theta_1>\sigma$, in which case its stationary distribution is geometric with parameter $1-p$, i.e. $p_n^\Lambda = (1-p)\,p^{n-1}$, $n\in\Nb$, where
\begin{equation}\label{pargeo}
p:=\left\{\begin{array}{ll}
                             \frac{\sigma}{\sigma+\theta_0}& \textrm{if $\theta_1=0$},\\
                      
                             \frac{\sigma+\theta-\sqrt{(\sigma-\theta)^2+4\sigma\theta_0}}{2\theta_1}& \textrm{if $\theta_1>0$}.
                           \end{array}\right.
                           \end{equation}
In this section, we aim to characterise the measures $\Lambda$ such that $L_\infty^\Lambda$ is geometrically distributed.  
\begin{proposition}\label{geomchar}
Let $\rho\in(0,1)$. The following assertions are equivalent
\begin{enumerate}
 \item The random variable $L_\infty^\Lambda$ is geometrically distributed with parameter $1-\rho$.
 \item $m_0=0$ and for all $n\in\Nb$,
  \begin{equation}\label{cg1}
\frac{1}{n}\int_{(0,1]}\left[(1-\rho)(1-x)^n+\rho-\left(\frac{1-x}{1-\rho x}\right)^n\right]\frac{\Lambda_0({\rm d}x)}{x^2} =\theta_1 \rho^2-(\sigma+\theta)\rho +\sigma.
 \end{equation}
 \item $m_0=m_1=0$ and for all $n\in\Nb_0$,
 \begin{equation}\label{cg2}
(1-\rho) \int_{(0,1)}\left[\left(\frac{1-x}{1-\rho x}\right)^n\frac{1}{1-\rho x}-(1-x)^n\right]\frac{\Lambda_0({\rm d}x)}{x}=\theta_1 \rho^2-(\sigma+\theta)\rho +\sigma.
\end{equation}
\item $m_0=m_1=0$ and for all $n\in\Nb_0$,
\begin{equation}\label{cg3a}
 \int_{(0,1)}(1-x)^n\Lambda_0({\rm d}x)=(1-\rho)\int_{(0,1)}\left(\frac{1-x}{1-\rho x}\right)^n\frac{\Lambda_0({\rm d}x)}{(1-\rho x)^2},
\end{equation}
and
\begin{equation}\label{cg3b}
\int_{(0,1)}\frac{\Lambda_0({\rm d}x)}{1-\rho x}=\frac{\theta_1 \rho^2-(\sigma+\theta)\rho +\sigma}{\rho(1-\rho)}.
\end{equation}
\end{enumerate}
If in addition $\Lambda\neq 0$, then $\int x^{-1}\Lambda_0({\rm d}x)=+\infty$, corresponding to a dust-free component.
\end{proposition}
\begin{proof}
From Eq. \eqref{frlm}, we see that $L_\infty^\Lambda\sim\textrm{Geom}(1-\rho)$ if and only if
  \begin{equation}\label{geoc1}
 \frac{1}{n}\sum\limits_{k=2}^\infty\frac{(n)_k^\uparrow}{k!}\lambda_{k+n,k}(\rho-\rho^k)+ \frac{m_1}{n}\rho 
 =\theta_1 \rho^2-(\sigma+\theta)\rho +\sigma,\quad n\in\Nb,
\end{equation}
 where $()^{\uparrow}$ is the rising factorial (see Appendix \ref{Ab}). Now, let us define
 \begin{equation*}
 \begin{split}
 I_n&:=\frac{1}{n}\int_{(0,1)}\left[(1-\rho)(1-x)^n+\rho-\left(\frac{1-x}{1-\rho x}\right)^n\right]\frac{\Lambda_0({\rm d}x)}{x^2},\quad n\in\Nb,\\
 J_n&:=(1-\rho) \int_{(0,1)}\left[\left(\frac{1-x}{1-\rho x}\right)^n\frac{1}{1-\rho x}-(1-x)^n\right]\frac{\Lambda_0({\rm d}x)}{x},\quad n\in\Nb_0,\\
 K_n&:= \int_{(0,1)}\left[(1-x)^n-\left(\frac{1-x}{1-\rho x}\right)^n\frac{1-\rho}{(1-\rho x)^2}\right]\Lambda_0({\rm d}x),\quad n\in\Nb_0.\\
 \end{split}
\end{equation*}
Using Fubini's theorem, we see that
\begin{equation}\label{aux0}
 \frac{1}{n}\sum\limits_{k=2}^\infty\frac{(n)_k^\uparrow}{k!}\lambda_{k+n,k}(\rho-\rho^k)=\frac{m_0(n+1)}{2}(\rho-\rho^2)+I_n.
\end{equation}
Therefore, $L_\infty^\Lambda\sim\textrm{Geom}(1-\rho)$ if and only if
\begin{equation}\label{aux1}
 \frac{m_0(n+1)}{2}(\rho-\rho^2)+I_n+\frac{m_1}{n}\rho=\theta_1 \rho^2-(\sigma+\theta)\rho +\sigma,\quad n\in\Nb.
\end{equation}
Since the left-hand side of \eqref{cg1} equals $I_n+m_1\rho/n$, clearly (2) implies (1). Now assume that $L_\infty^\Lambda\sim\textrm{Geom}(1-\rho)$. A straightforward application of Fubini's theorem shows that
$$I_n= \frac{1}{n}\sum\limits_{k=2}^\infty\frac{(n)_k^\uparrow}{k!}(\rho-\rho^k)\int_{(0,1)}x^k(1-x)^n\frac{\Lambda_0({\rm d}x)}{x^2}\geq 0.$$
Therefore, if $m_0>0$, the left-hand side of \eqref{aux1} tends to infinity as $n\to\infty$, in contrast to the right-hand side which is constant. We conclude that $m_0=0$. Hence Eq. \eqref{aux1} yields \eqref{cg1}. Thus (1) implies (2).

Now we prove that (2) implies (3). Indeed, if (2) holds true, then
\begin{equation}\label{aux2}
 kI_k+m_1\rho=k(\theta_1 \rho^2-(\sigma+\theta)\rho +\sigma),\quad k\in\Nb.
\end{equation}
Note that $J_n=(n+1)I_{n+1}-n I_n$. Therefore, \eqref{cg2}, for $n\geq 1$, is obtained by writing down Eq. \eqref{aux2} for $k=n$ and $k=n+1$ and taking the difference of these two equations. In addition, since 
\begin{equation}\label{aux3}
 nI_n=I_1+\sum_{k=1}^{n-1}J_n, \quad n\in\Nb,
\end{equation}
we deduce that
$n I_n=n(\theta_1 \rho^2-(\sigma+\theta)\rho +\sigma).$
Comparing this equation with \eqref{aux2}, we get $m_1=0$. Since for $m_1=0$ we have $I_1=J_0$, we conclude that \eqref{cg1} holds true also for $n=0$, which ends the proof that (2) implies (3). Moreover, since $I_1=J_0$ for $m_1=0$, (2) follows directly from (3) using \eqref{aux3}. 

It remains to prove the equivalence of (3) and (4), but this follows using that
$$J_n-J_{n+1}=(1-\rho)K_n,\quad J_n=J_1-(1-\rho)\sum_{k=0}^{n-1}K_k,\quad\textrm{and}\quad J_0=(1-\rho)\rho\int_{(0,1)}\frac{\Lambda_0({\rm d}x)}{1-\rho x}.$$
Finally, let us assume that $L_\infty^\Lambda\sim\textrm{Geom}(1-\rho)$ and that  $\int x^{-1}\Lambda_0({\rm d}x)<+\infty$. Applying the dominated convergence theorem, we get that
the left-hand side of \eqref{cg2} converges to zero as $n\to\infty$. Hence the right-hand side of \eqref{cg2} has to be zero. Since the function integrated in \eqref{cg2} is non-negative, we conclude that it has to be zero, which is impossible. This finishes the proof.
\end{proof}
A first consequence of the previous result is that for the Wright--Fisher diffusion model and for the star-shaped model the distribution of $L_\infty^\Lambda$ is not geometric. 
Next, we show the existence of a non-trivial $\Lambda$ measure such that $L_\infty^\Lambda$ has the geometric distribution. More precisely, we show that the uniform measure on $[0,1]$ provides such an example.
\begin{lemma}\label{uniquerho}
Let $\Lambda$ be the uniform measure on $[0,1]$. Then, there is a unique $\rho:=\rho(\sigma, \theta_0,\theta_1)\in(0,1)$ satisfying Eq. \eqref{cg3b}. Moreover, 
\begin{enumerate}
 \item if $\theta_0=\theta_1=0$, then $\rho=1-e^{-\sigma}$.
 \item if $\theta_0=0$ and $\theta_1>0$, then $\rho=1-\theta_1^{-1}\,W\left(\theta_1 e^{\theta_1-\sigma}\right)$.
 \item if $\theta_0>0$ and $\theta_1=0$, then $\rho=1-\theta_0\, \left[W\left(\theta_0 e^{\theta_0+\sigma}\right)\right]^{-1}$,
\end{enumerate}
where $W$ denotes the (single-valued) restriction to $\Rb_+$ of the (multi-valued) Lambert-W function (see, e.g., \cite{CGHJK96}). 
\end{lemma}
\begin{proof}
 A straightforward calculation shows that \eqref{cg3b} is equivalent to $(\sigma+\log(1-\rho)-\theta_1 \rho)(\rho-1)+\theta_0 \rho=0$. Therefore, we only need to show that the function $r:(0,1)\to\Rb$ defined via $$r(x):=(\sigma+\log(1-x)-\theta_1 x)(x-1)+\theta_0 x,\quad x\in (0,1),$$ 
has a unique root. For this note that for all $x\in(0,1)$, we have $r''(x)=-2\theta_1-\frac1{1-x}<0$, and hence $r'$ is strictly decreasing in $(0,1)$. Moreover, $r'(0+)=1+\sigma+\theta>0$ and $r'(1-)=-\infty$. We infer that $r'$ has a unique root $x_0\in(0,1)$ and that $r$ is strictly decreasing in $(x_0,1)$. In addition, $r(1-)=\theta_0\geq 0$, and thus, $r(x)>\theta_0$ for $x\in(x_0,1)$. Since, $r(0+)=-\sigma<0$, this implies that $r$ has a root in $(0,1)$. The uniqueness of this root is a consequence of the strict monotonicity of $r'$. 
 
It remains to show the explicit formulas for $\rho$ in the cases (1), (2) and (3). Case (1) is trivial. Cases (2) and (3) follow using that the function $W$ is the inverse of the function $x\mapsto xe^x$.
\end{proof}
\begin{corollary}\label{geomlawBS}
For the Bolthausen--Sznitman model with selection parameter $\sigma>0$ and mutation parameters $\theta_0,\theta_1\geq 0$, the stationary distribution of the block counting process is geometric with parameter $1-\rho$, where $\rho$ is the unique
solution of \eqref{cg3b}. 
\end{corollary}
\begin{proof}
 Using the change of variables $y=(1-\rho)x/(1-\rho x)$, we get
 $$(1-\rho)\int_0^1\left(\frac{1-x}{1-\rho x}\right)^n\frac{{\rm d}x}{(1-\rho x)^2}=\int_0^1 (1-y)^n {\rm d}y,$$
 and therefore the uniform measure on $[0,1]$ satisfies \eqref{cg3a}. Since \eqref{cg3b} is satisfied for $\rho$, the result follows using Proposition \ref{geomchar}.
\end{proof}

\begin{corollary}\label{absBS}
 For the Bolthausen--Sznitman model with selection parameter $\sigma>0$ and no mutation, i.e. $\theta=0$, we have
 $$P_x(X_\infty=0)=\frac{(1-x)e^{-\sigma}}{x+(1-x)e^{-\sigma}}, \qquad x\in[0,1].$$
\end{corollary}
\begin{proof}
Since $R$ and $L^\Lambda$ have the same distribution when $\theta=0$, the result follows from Proposition \ref{mdlwfm}, Lemma \ref{uniquerho} and Corollary \ref{geomlawBS}.  
\end{proof}

It seems natural to ask if the uniform measure on $[0,1]$ is the unique (up to a multiplicative constant) measure $\Lambda$ leading to the geometric distribution. This question will be the matter in the rest of this section. 
\begin{lemma}\label{transf}
Let $\varphi:[0,1]\to [0,1]$ be defined via $\varphi(x):=(1-x)/(1-\rho x)$, $x\in[0,1]$. If $L_\infty^\Lambda\sim\textrm{Geom}(1-\rho)$, then the moments $y_k:=\int y^{k}\mu({\rm d}y)$, $k\in\Nb_0$, where $\mu:=\Lambda\circ\varphi^{-1}$ denotes the pushforward of the measure $\Lambda$ by $\varphi$, satisfy the linear system of equations 
 $$y_n-2\rho y_{n+1}+\rho^2 y_{n+2}-(1-\rho)^{n+1}\sum_{k\geq 0}\binom{n+k-1}{k}\rho^k y_{n+k}=0,\quad n\in\Nb,$$
 and $\rho y_0-2\rho y_1+\rho^2 y_2=0.$
 \end{lemma}
 \begin{proof}
  Noting  that $\varphi$ is a M\"obius transformation satisfying $\varphi^{-1}=\varphi$, a straightforward calculation shows that \eqref{cg3a} translates into
  $$(1-\rho)^n \int_{[0,1]}\frac{y^n}{(1-\rho y)^n}\mu({\rm d}y)=\frac{1}{1-\rho}\int_{[0,1]}y^n(1-\rho y)^2\mu({\rm d}y), \quad n\in\Nb_0.$$
  Moreover,
  $$\int_{[0,1]}y^n(1-\rho y)^2\mu({\rm d}y)=y_n-2\rho y_{n+1}+\rho^2 y_{n+2},$$
  and
  $$\int_{[0,1]}\frac{y^n}{(1-\rho y)^n}\mu({\rm d}y)=\int_{[0,1]}y^n\sum_{k\geq 0}\binom{n+k-1}{k}(\rho y)^k\mu({\rm d}y)=\sum_{k\geq 0}\binom{n+k-1}{k}\rho^k y_{n+k}.$$
  The result follows.
 \end{proof}

Let us consider the linear operator $S:\ell^\infty\to\ell^\infty$ on the Banach space $\ell^\infty:=\{x=(x_i)_{i\in\Nb_0}\in\Rb^\infty :  \rVert x\lVert:=\sup_{i\in\Nb_0}|x_i|<\infty\}$ defined via
$$(Sx)_n:=2 \rho x_{n+1}-\rho^2x_{n+2}+(1-\rho)^{n+1}\sum_{k\geq 0}\binom{n+k-1}{k}\rho^k x_{n+k}\quad\textrm{and}\quad (Sx)_0:=2\rho x_1-\rho^2 x_2+(1-\rho)x_0,$$
for all $n\in\Nb$ and $x\in\ell^\infty$. Note that from Lemma \ref{transf}, if $L_\infty^\Lambda\sim\textrm{Geom}(1-\rho)$, then the vector $y:=(y_k)_{k\in\Nb_0}$ of the moments of $\mu=\Lambda\circ \varphi^{-1}$ is a fixed point of $S$.
We are then interested on the fixed points of $S$ arising as the moments of a finite measure. It is a well known result of Hausdorff \cite{Ha21} that a non-negative sequence $x:=(x_k)_{k\in\Nb_0}$ corresponds to the moments of a finite measure if and only if $x$ is a completely monotone sequence, i.e. if 
$\Delta^n x_k:=\Delta^{n-1} x_k-\Delta^{n-1}x_{k+1}\geq 0$ for all $k\in\Nb_0$ and $n\in\Nb$,
where $\Delta^0$ is the identity operator.
The set $K:=\{x\in\ell^\infty: \textrm{$x$ is completely monotone}\}$ is a closed convex cone in $\ell^\infty$. In particular, the set $X:=K-K=\{x-y: x,y\in K\}$ is a Banach space. The latter is known as the set of moment sequences, since its elements are exactly the sequences that are obtained as the moments of a finite signed measure on $[0,1]$. We aim to determine the dimension of the set of fixed points of $S$ in $X$.

\begin{proposition}\label{p1}
Let $\mu$ be a finite measure on $[0,1]$ and let $x:=(x_k)_{k\in\Nb_0}\in K$ be the sequence of moments of $\mu$, then
\begin{equation}\label{smu}
(Sx)_k=\int_{[0,1]} y^k \mu_S({\rm d}y), \quad k\in\Nb_0,
\end{equation}
where $\mu_S ({\rm d}y) :=\rho y(2-\rho y)\mu({\rm d}y) +(1-\rho)\mu\circ \phi^{-1}({\rm d}y)$ and $\phi:[0,1]\to[0,1]$ is defined via
$$\phi(y):=\frac{(1-\rho)y}{1-\rho y},\quad y\in[0,1].$$
In particular, $S(K)\subset K$. 
\end{proposition}
\begin{proof}
 From definition, we have
 $$(Sx)_k=\int_{[0,1]}y^k\left(2\rho y-\rho^2 y^2 +(1-\rho)\left(\frac{1-\rho}{1-\rho y}\right)^k \right)\mu({\rm d}y),\quad k\in\Nb_0.$$
 The first result follows. The second one is a direct consequence of \eqref{smu}.
\end{proof}
We can now identify the restriction of $S$ to $K$ with the operator $\Ss:\Ms_f([0,1])\to\Ms_f([0,1])$ on the space of finite measures defined via
\begin{equation}\label{Smu}
\Ss(\mu)({\rm d}y)=\rho y(2-\rho y)\mu({\rm d}y) +(1-\rho)\mu\circ \phi^{-1}({\rm d}y),\quad \mu\in\Ms_f([0,1]).
\end{equation}
Moreover, fixed points of $S$ in $K$ are in a one-to-one relation with fixed points of $\Ss$. Note that if $\mu$ is a fixed point of $\Ss$ then its support is invariant under $\phi$. 

We denote by $\Ms_f^{ac}$, $\Ms_f^{sc}$ and $\Ms_f^{d}$ the subsets of finite measures that are absolutely continuous, singular continuous and discrete, respectively. We know that if $\Lambda$ is the uniform distribution on $[0,1]$, then $L_\infty^\Lambda\sim\textrm{Geom}(1-\rho)$. Therefore, using Lemma \ref{transf} and Proposition \ref{p1}, we conclude that the measure $\mu\in\Ms_f^{ac}$ with density $h(y):=(1-\rho)/(1-\rho y)^2$, $y\in[0,1]$, is a fixed point of $\Ss$. 

For each $k\in\Nb$, $\phi^{(k)}$ denotes the $k$-th iteration of the function $\phi$, and $\phi^{(-k)}$ denotes its inverse.
\begin{lemma}\label{iterates}
 For all $n\in\Nb$ and $x\in[0,1]$, we have
 $$\phi^{(n)}(x)=\frac{(1-\rho)^nx}{1-x(1-(1-\rho)^n)}\quad\textrm{and}\quad \phi^{(-n)}(x)=\frac{x}{(1-\rho)^n+x(1-(1-\rho)^n)}.$$
\end{lemma}
\begin{proof}
The first identity can be shown by induction. The second one follows from the first one.  
\end{proof}
\begin{proposition}
Let $\mu\in\Ms_f^d$ be a fixed point of $\Ss$, then $\mu(\{0\})=\mu(\{1\})=0$. Moreover, if $x_0\in(0,1)$ has positive mass, then for all $k\in\Zb$,
$m_{k}:=\mu(\{\phi^{(k)}(x_0)\})>0$ and
\begin{equation}\label{mk}
m_{-k}=\frac{(1-\rho)^{k-2}(1-\rho x_0)^2\, m_0}{((1-\rho)^{k-1}+x_0(1-(1-\rho)^{k-1}))^2}\quad\textrm{and}\quad m_{k}=\frac{(1-\rho)^k(1-\rho x_0)^2\,m_0}{(1-x_0(1-(1-\rho)^{k+1}))^2}, \quad k\in\Nb.
\end{equation}
\end{proposition}
\begin{proof}
Let $\mu\in\Ms_f^d$ be a fixed point of $\Ss$ and assume that there is $x_0\in[0,1]$ with $m_0:=\mu(\{x_0\})>0$. From \eqref{Smu} we deduce that
$m_0=\rho x_0(2-\rho x_0)m_0+ (1-\rho)\mu (\{\phi^{(-1)}(x_0)\}).$
Therefore, $x_0\in(0,1)$ and  $m_{-1}:=\mu(\{\phi^{-1}(x_0)\})>0$. Iterating the argument yields, for all $k\in\Nb$, $m_{-k}:=\mu(\{\phi^{(-k)}(x_0)\})>0$ and
\begin{equation}\label{r1a}
 m_{-k}\frac{\left(1-\rho \phi^{(-k)}(x_0)\right)^2}{1-\rho}=m_{-k-1}, \quad k\in\Nb.
\end{equation}
Hence, $m_{-k}=m_0\prod_{i=0}^{k-1}(1-\rho\phi^{(-i)}(x_0))^2/(1-\rho)$. The first identity follows using that
$$1-\rho\phi^{(-i)}(x_0)=(1-\rho)\frac{(1-\rho)^{i-1}+x_0(1-(1-\rho)^{i-1})}{(1-\rho)^{i}+x_0(1-(1-\rho)^{i})},\quad i\in\Nb_0.$$
Similarly, for $k\in\Nb$, $m_{k}:=\mu(\{\phi^{(k)}(x_0)\})>0$ and
\begin{equation}\label{r2a}
 m_{k+1}=\frac{1-\rho}{\left(1-\rho \phi^{(k+1)}(x_0)\right)^2}m_{k}, \quad k\in\Nb.
\end{equation}
Thus, $m_{k}=m_0\prod_{i=1}^{k}(1-\rho)/(1-\rho\phi^{(i)}(x_0))^2$. The second identity follows using that
$$1-\rho\phi^{(i)}(x_0)=\frac{1-x_0(1-(1-\rho)^{i+1})}{1-x_0(1-(1-\rho)^{i})},\quad i\in\Nb_0.$$
\end{proof}
The next result provides a class of fixed points of $\Ss$ in $\Ms_f^d$ and of discrete measures $\Lambda$ such that $L_\infty^\Lambda$ is geometrically distributed.
\begin{proposition}
 For any $\rho, x_0\in(0,1)$ and $m_0>0$, the measure $\mu:=\mu(\rho,x_0,m_0)$ given by
 $$\mu:=\sum_{k\in\Zb}m_k \delta_{\phi^{(k)}(x_0)},$$
 where the coefficients $(m_k)_{k\in\Zb}$ are defined via Eq. \eqref{mk}, is a fixed point of $\Ss$ in $\Ms_f^d$. In addition, for
 $m_0<\sigma x_0(1-x_0)$, the equation $m_0(1-\rho x_0)^2-x_0(1-x_0)(\theta_1 \rho^2-(\sigma+\theta)\rho+\sigma)=0$ has a unique solution $\rho_*$ in $(0,1)$. Setting $\Lambda:=\mu(\rho_*,x_0,m_0)\circ \varphi^{-1}$, we have $L_\infty^\Lambda\sim\textrm{Geom}(1-\rho_*)$. 
\end{proposition}
\begin{proof}
 Since $\lim_{k\to\infty} \phi^{(k)}(x_0) =0$ and $\lim_{k\to\infty} \phi^{(-k)}(x_0) =1$, then there exists $k_0\in\Nb$ such that 
$$m_k\leq c_{x_0}\left(1-\frac{\rho}{2}\right)^k m_0\quad\textrm{and}\quad m_{-k}\leq C_{x_0}\left(1-\frac{\rho}{2}\right)^k m_0,\quad k>k_0,$$
for some appropriate constants $c_{x_0},C_{x_0}>0$. Therefore, $\mu\in\Ms_f^d$. Moreover, since the coefficients $(m_k)_{k\in\Zb}$ satisfy \eqref{r1a} and \eqref{r2a}, it follows that for all $k\in\Zb$,
$\Ss\mu (\{\phi^{(-k)}(x_0)\})=\mu(\{\phi^{(-k)}(x_0)\}).$
Hence, $\mu$ is a fixed point of $\Ss$. Now, we assume that $m_0<\sigma x_0(1-x_0)$. Note that the function $r:[0,1]\to\Rb$ defined via
$r(z):=m_0(1-z x_0)^2-x_0(1-x_0)(\theta_1 z^2-(\sigma+\theta)z+\sigma)$, $z\in[0,1],$
is a quadratic polynomial with $r(0)=m_0-\sigma x_0(1-x_0)<0$ and $r(1)=m_0(1-x_0)^2+ x_0(1-x_0)\theta_0>0$, and therefore, it has a unique root $\rho_*$ in $(0,1)$. Let $\mu:=\mu(\rho_*,x_0,m_0)$. Using that $\mu$ is a fixed point of $\Ss$, one can easily show that the measure $\Lambda:=\mu\circ \varphi^{-1}$ satisfy \eqref{cg3a} for $\rho=\rho_*$. It remains to show that $\Lambda$ satisfies \eqref{cg3b} for $\rho=\rho_*$. Using that $\varphi=\varphi^{-1}$ and the definition of $\mu$, we obtain
$$\int_{(0,1)}\frac{\Lambda_0({\rm d}x)}{1-\rho_* x} =\sum_{i\in\Zb} a_i\quad \textrm{where}\quad a_i=m_i \frac{1-\rho_* \phi^{(i)}(x_0)}{1-\rho_*},\quad i\in\Zb.$$
Moreover, setting $b_n:=1-x_0(1-(1-\rho_*)^n)$ for $n\in\Nb_0$,  we get
\begin{align*}
 a_n&=\frac{m_0(1-\rho_* x_0)^2 (1-\rho_*)^{n-1}}{b_{n+1}b_n}=\frac{m_0(1-\rho_* x_0)^2}{\rho_*(1-x_0)}\left(\frac{(1-\rho_*)^{n-1}}{b_n}-\frac{(1-\rho_*)^n}{b_{n+1}}\right).
\end{align*}
Hence, $$\sum_{i\in\Nb_0}a_i=\frac{m_0(1-\rho_* x_0)^2}{\rho_*(1-\rho_*)(1-x_0)}.$$
Similarly, setting $c_n:=(1-\rho_*)^n+x_0(1-(1-\rho_*)^n)$, $n\in\Nb$, we obtain
\begin{align*}
 a_{-n}&=\frac{m_0(1-\rho_* x_0)^2 (1-\rho_*)^{n-2}}{c_{n-1}c_n}=\frac{m_0(1-\rho_* x_0)^2}{\rho_*x_0}\left(\frac{(1-\rho_*)^{n-2}}{c_{n-1}}-\frac{(1-\rho_*)^{n-1}}{c_{n}}\right).
\end{align*}
Therefore,
$$\sum_{i\in\Nb}a_{-i}=\frac{m_0(1-\rho_* x_0)^2}{\rho_*(1-\rho_*)x_0}.$$
Summarising, we have
$$\int_{(0,1)}\frac{\Lambda_0({\rm d}x)}{1-\rho_* x} =\sum_{i\in\Zb} a_i=\frac{m_0(1-\rho_* x_0)^2}{\rho_*(1-\rho_*)x_0(1-x_0)}=\frac{\theta_1 \rho_*^2-(\sigma+\theta)\rho_*+\sigma}{\rho_*(1-\rho_*)},$$
ending the proof.

\end{proof} 
As a consequence the dimension of the set of fixed points of $S$ in $X$ is infinite. In the next proposition, we show that the measure $\mu({\rm d}y)=h(y){\rm d}y$, with $h(y)=1/(1-\rho y)^2$, $y\in[0,1]$, is the unique fixed point of $\Ss$ (up to a multiplicative constant) in $\Ms_f^{ac}$ having a density which is continuous in $[0,1]$. 
\begin{proposition}
Let $h:[0,1]\to\Rb_+$ be a continuous function on $[0,1]$. The measure $\mu({\rm d}y)=h(y){\rm d}y$ on $[0,1]$ is a fixed point of $\Ss$ if and only if 
$$h(y)=\left(\frac{1-\rho}{1-\rho y}\right)^2 h(1), \quad y\in[0,1].$$
\end{proposition}
\begin{proof}
Proving that the function $h$ defined in the statement leads to a fixed point of $\Ss$ is straightforward. Now, assume that $\mu({\rm d}y)=h(y){\rm d}y$ is a fixed point of $\Ss$. It follows that
$$h(y)=\frac{1-\rho}{(1-\rho y)^2} h(\phi^{(-1)}(y)){\phi^{(-1)}}^{\prime}(y)=\frac{(1-\rho)^2}{(1-\rho y)^2(1-\rho+\rho y)^2}h(\phi^{(-1)}(y)), \quad y\in(0,1).$$
Iterating this equation, we obtain 
\begin{equation}\label{prod}
 h(y)=\left(\prod_{k=0}^n\frac{1-\rho}{(1-\rho \phi^{(-k)}(y))(1-\rho+\rho \phi^{(-k)}(y))}\right)^2 h(\phi^{-{(n+1)}}(y)),\quad y\in(0,1).
\end{equation}
Using Lemma \ref{iterates}, we deduce that
$$\frac{1-\rho}{(1-\rho \phi^{(-k)}(y))(1-\rho+\rho \phi^{(-k)}(y))}=\frac{a_k^2(y)}{a_{k+1}(y) a_{k-1}(y)},\quad k\in\Nb_0,$$
where $a_k(y):=(1-\rho)^k+y(1-(1-\rho)^k)$, $k\geq -1$. Hence,
$$\prod_{k=0}^n\frac{1-\rho}{(1-\rho \phi^{(-k)}(y))(1-\rho+\rho \phi^{(-k)}(y))}=\frac{a_0(y)}{a_{-1}(y)}\frac{a_n(y)}{a_{n+1}(y)}=\frac{1-\rho}{1-\rho y}\frac{a_n(y)}{a_{n+1}(y)}.$$
Letting $n\to\infty$ in \eqref{prod} yields the result.
\end{proof}
\begin{remark}\label{remgeo}
Consider a measure of the form $\mu({\rm d}y)=h(y){\rm d}y$ with $h:[0,1]\to\Rb_+$ being measurable in $[0,1]$ and continuous in an open interval $I\subset (0,1)$. On can easily check that if $\mu$ is a fixed point of $\Ss$, then $h$ is continuous in $(0,1)$. However, for the uniqueness (up to a multiplicative constant) of such a fixed point, the continuity of $h$ at $1$ is crucial. Indeed, let $x_0\in(0,1)$ and let $p:[0,1]\to\Rb_+$ be a continuous function on $[0,1]$ such that $p(0)=p(1)$. We define the function $C:(0,1)\to\Rb_+$ via
$$C(x):=p\left(\frac{\phi^{(k)}(x)-x_0}{\phi(x_0)-x_0}\right), \quad\textrm{for} \quad x\in[\phi^{(-k)}(x_0),\phi^{(-k+1)}(x_0)),\, k\in\Zb.$$
The function $C$ is bounded, continuous in $(0,1)$ and constant on each set of the form $\{\phi^{(i)}(y):i\in\Zb\}$ for some $y\in(0,1)$. Hence, the function $h:(0,1)\to\Rb_+$ defined via $h(y):=C(y)/(1-\rho y)^2$, $y\in(0,1)$, satisfies
$$h(y)=\frac{1-\rho}{(1-\rho y)^2} \,h(\phi^{(-1)}(y))\,{\phi^{(-1)}}^{\prime}(y),\quad y\in(0,1).$$
Therefore, the measure $\mu({\rm d}y):=h(y){\rm d}y$ is a fixed point of $\Ss$. Moreover, one can prove that all the fixed points of $\Ss$ in $\Ms_f^{ac}$ having a density which is continuous in $(0,1)$ are of this form. As a consequence, the unique $\beta(a,b)$-model leading to a geometric distribution is the $\beta(1,1)$-model, i.e. the Bolthausen--Sznitman model.
\end{remark}

\section{Solving the Fearnhead-type recursion for the Moran model}\label{s4}
In the Moran model with mutation and selection, the stationary distribution of the block counting process is characterised by Equations \eqref{frmm} and \eqref{bcmm}, which, using that $a_n^N=\sum_{k=n+1}^N p_k^N$, turn out to be equivalent to  
\begin{equation}\label{prmm}
 \left(\frac{n}{N}+u_1\right)p_n^N=\frac{(N-n+1)s}{N}p_{n-1}^N-u_0\sum\limits_{\ell=n}^N p_\ell^N,\quad n\in\{2,\ldots,N-1\},
\end{equation}
together with the boundary conditions
\begin{equation}\label{pbcmm}
\sum\limits_{i=1}^N p_i^N=1 \quad\textrm{and}\quad (1+u)p_N^N=\frac{s}{N}p_{N-1}^N.
\end{equation}
 Let $D:=\{z\in\Cb: |z|<1\}$ denote the open unit disk and set $D_*=D\setminus\{z: \textrm{Im}(z)=0,\, \textrm{Re}(z)\leq 0\}$. For $z_1,z_2\in D_*$ and any holomorphic function $f:D_*\rightarrow \Cb$ we denote by $\int_{z_1}^{z_2} f(\xi){\rm d}\xi$ the integral of $f$ along any smooth path  in $D_*$ connecting $z_1$ and $z_2$. The following is the main result of this section. It provides explicit expressions for the stationary distribution $(p_n^N)_{n\in[N]}$ and its probability generating function $g_N:\Cb\rightarrow\Cb$, defined via $g_N(z):=\sum_{n=1}^{N} p_n^N\, z^n.$
\begin{theorem}\label{mainMM}
For the Moran model with selection parameter $s>0$ and mutation parameters $u_0,u_1\geq 0$ the following holds
\begin{itemize}
 \item[(i)] If $u_0=0$, then 
\begin{equation}\label{pmfMMu0}
p_n^N=\frac{1}{\pFq{2}{1}{1;,1-N}{Nu+2}{-s}}\frac{(N-1)_{n-1}^{\downarrow}}{(Nu+2)_{n-1}^{\uparrow}}\,s^{n-1},\quad n\in[N],
\end{equation}
where $\hf$ is the Gauss hypergeometric function (see Appendix \ref{Ab}). In particular, we have
\begin{equation}\label{pgfMMu0}
  g_N(z)=\frac{\pFq{2}{1}{1;,1-N}{Nu+2}{-sz}}{\pFq{2}{1}{1;,1-N}{Nu+2}{-s}},\quad z\in \Cb.
\end{equation}
 \item[(ii)] If $u_0>0$, then for all $z\in D_*$
\begin{equation}\label{pgfMM}
 g_N(z)=\frac{N u_0 \,I_0^N}{s(I_0^N-I_1^N)}\frac{\left(z+\frac{1}{s}\right)^{(1+u_1+\rho_0)N}}{z^{Nu_1}\,(1-z)^{N\rho_0}}\int_{0}^z\frac{\left(\frac{I_1^N}{I_0^N}-\xi\right)\,\xi^{Nu_1}\,(1-\xi)^{N\rho_0-1}}{\left(\xi+\frac{1}{s}\right)^{(1+u_1+\rho_0)N+1}}\,{\rm d}\xi,
\end{equation}
with $\rho_0:=u_0/(1+s)$ and
$I_i^N:=\int_{0}^1\frac{y^{Nu_1+i}\,(1-y)^{N\rho_0-1}}{(y+\frac{1}{s})^{(1+u_1+\rho_0)N+1}}\,{\rm d}y$, $i\in\{0,1\}$.
Moreover, 
\begin{equation}\label{pmfMM}
 p_n^N=\frac{Nu_0 }{I_0^N-I_1^N}\left[\frac{I_1^N}{Nu_1+1}q_{n,1}^{N}-\frac{I_0^N}{Nu_1+2}q_{n,2}^{N}\right], \quad n\in[N],
\end{equation}
 where $q_{1,1}^{N}:=1$, $q_{1,2}^{N}:=0$, and for $n\in\{2,\ldots,N\}$
\begin{align*}
q_{n,i}^{N}&:=\sum\limits_{m=0}^{n-i}\frac{(-N+i-1)_m^\uparrow}{(Nu_1+i+1)_m^\uparrow}  (-s)^m\pFq{3}{2}{m+1;,1-N\rho_0;,m-n+i}{Nu_1+m+i+1;,1}{1+s},\quad i\in\{0,1\},
\end{align*}
where $\HF$ is the generalised hypergeometric function (see Appendix \ref{Ab}). 
\end{itemize}
\end{theorem}
\begin{remark}
When $u=0$, Eq. \eqref{pmfMMu0} implies that $L_\infty^N$ is a binomial random variable with parameters $N$ and $s/(1+s)$ conditioned to be strictly positive (see also \cite[Sect.~3.1]{FC17}). Hence, \cite[Prop.~4.7]{FC17} leads to the following expression for the absorption probability of $X^N$ at $N$ (c.f \cite[Chap. 6.1.1]{D08})
$$P_k(X^N_\infty=N)=\frac{(1+s)^{N}-(1+s)^{N-k}}{(1+s)^N-1}.$$
\end{remark}
The proof of Theorem \ref{mainMM} is based on the following result.
\begin{lemma} The generating function $g_N $ satisfies the ordinary differential equation
\begin{align}\label{edomm}
 z(1-z)(1+sz)\,g_N'(z)&=-N(sz^2-(s+u)z+u_1)\,g_N(z)+\left(1+N u_1\right)p_1^N\,z(1-z)-Nu_0\, z^2,
\end{align}
on $D$, with boundary conditions $g_N(0)=0$ and $g_N(1)=1$.
\end{lemma}
\begin{proof}
The boundary conditions follow from the definition of $g_N$. Multiplying \eqref{prmm} with $z^n$ and summing over all $n\in\{2,\ldots,N-1\}$ yields
$$ \sum\limits_{n=2}^{N-1}\left(\frac{n}{N}+u_1\right)p_n^N z^n=\frac{s}{N}\sum\limits_{n=2}^{N-1}(N-n+1)p_{n-1}^N z^n-u_0\sum\limits_{n=2}^{N-1} z^n\sum\limits_{\ell=n}^N p_\ell^N.$$
The left hand side is equal to
$$\sum\limits_{n=2}^{N-1}\left(\frac{n}{N}+u_1\right)p_n^N z^n=\frac{z}{N}g_N'(z)+u_1 g_N(z)-\left(\frac{1}{N}+u_1\right)p_1^N z- (1+u_1)p_N^N z^N.$$
Moreover,
$$\frac{s}{N}\sum\limits_{n=2}^{N-1}(N-n+1)p_{n-1}^N z^n=-\frac{sz^2}{N}g_N'(z)+sz\,g_N(z)-\frac{s}{N}p_{N-1}^N \,z^N,$$
$$\sum\limits_{n=2}^{N-1} z^n\sum\limits_{\ell=n}^N p_\ell^N=-\frac{z}{1-z}g_N(z)-p_N^N z^N +\frac{z^2}{1-z}.$$
The result follows putting everything together and using Eq. \eqref{pbcmm}.
\end{proof}
\begin{proof}[Proof of Theorem \ref{mainMM}]
 {{\bf (i)}:} Formula \eqref{pmfMMu0} is obtained by iteration of \eqref{prmm} and using \eqref{pbcmm}. Formula \eqref{pgfMMu0} is a direct consequence of Eq. \eqref{pmfMMu0}.
 
  {{\bf (ii)}:} Since both sides of \eqref{pgfMM} are analytic in $D_*$, it suffices to show that they coincide on the real interval $(0,1)$. Thus, we have to solve \eqref{edomm} in $(0,1)$. Separation of variables in the homogeneous equation $x(1-x)(1+sx)\,g'(x)+N(sx^2-(s+u)x+u_1)\,g(x)=0$ on $(0,1)$ implies that its basic solution is given by
 $g(x)=(x+s^{-1})^{(1+u_1+\rho_0)N}x^{-Nu_1}(1-x)^{-N\rho_0}$, $x\in (0,1)$.
 Hence, for any $x_0\in (0,1)$, the variation of constants method yields
 $$g_N(x)=\frac{1}{s}\frac{\left(x+\frac1{s}\right)^{(1+u_1+\rho_0)N}}{x^{Nu_1}(1-x)^{N\rho_0}}\left[c_{x_0}-\int_{x}^{x_0}\frac{(\beta_N-\alpha_N\,\xi)\xi^{Nu_1}(1-\xi)^{N\rho_0-1}}{\left(\xi+\frac1{s}\right)^{(1+u_1+\rho_0)N+1}}{\rm d}\xi\right], \quad x\in(0,x_0),$$
where $c_{x_0}$ is a constant, $\alpha_N:=Nu_0+(1+Nu_1) p_1^N$ and $\beta_N:=(1+Nu_1)p_1^N$. Moreover, the boundary condition $g_N(0)=0$ implies that
\begin{equation}\label{prepgfMM}
g_N(x)=\frac{1}{s}\frac{\left(x+\frac1{s}\right)^{(1+u_1+\rho_0)N}}{x^{Nu_1}(1-x)^{N\rho_0}}\left[\,\int_{0}^{x}\frac{(\beta_N-\alpha_N\,\xi)\xi^{Nu_1}(1-\xi)^{N\rho_0-1}}{\left(\xi+\frac1{s}\right)^{(1+u_1+\rho_0)N+1}}{\rm d}\xi\right], \quad x\in (0,x_0).
\end{equation}
Since $x_0\in(0,1)$ is arbitrary, the previous identity holds for all $x\in(0,1)$. Letting $x\to 1$ and using that $g_N(1)=1$, we infer that $\beta_N I_0^N= \alpha_N I_1^N$. Hence, $p_1^N=Nu_0 I_1^N/((Nu_1+1)(I_0^N-I_1^N))$. Plugging the resulting expressions for $\alpha_N$ and $\beta_N$ in \eqref{prepgfMM} shows that \eqref{pgfMM} holds in $(0,1)$, and thus in $D_*$. 

Note that from \eqref{pgfMM} and Corollary \ref{IvsF1}, we have for all $z\in \{w\in D: |w|<1/\sqrt{1+2s}\}$
 \begin{equation}\label{pnf}
  g_N(z)=\frac{Nu_0 }{I_0^N-I_1^N}\left[\frac{I_1^N}{Nu_1+1}q_{N,1}(z)-\frac{I_0^N}{Nu_1+2}q_{N,2}(z)\right],
 \end{equation}
 where for $i\in\{1,2\}$
 $$q_{N,i}(z)=\frac{z^i\left(sz+1\right)^{N(1+\rho_0)-i}}{(1-z)^{N\rho_0}}\pFq{}{1}{Nu_1+i\,;,-N+i-1\,;,1-N\rho_0}{Nu_1+i+1}{\frac{z}{z+\frac1s}\,;,\frac{\left(1+\frac1s\right)z}{z+\frac1s}}.$$
 From \cite[Eq.~(25)]{A82}, we get for $z$ sufficiently small
 $$q_{N,i}(z)=\frac{z^i}{1-z}\pFq{}{1}{1\,;,-N+i-1\,;,1-N\rho_0}{Nu_1+i+1}{-sz\,;,\frac{-\left(1+s\right)z}{1-z}}.$$
 Since $(1-z)^{-m}=\sum_{k=0}^\infty \frac{(m)^{\uparrow}_k}{k!}z^k$, the series representation of the Appell function $F_1$ (see \eqref{def1}) yields
 $$q_{N,i}(z)=\sum\limits_{\ell=i}^\infty z^{\ell}\sum\limits_{m=0}^{\ell-i}\frac{(-N+i-1)_m^\uparrow}{(Nu_1+i+1)_m^\uparrow}  (-s)^m\pFq{3}{2}{m+1;,1-N\rho_0;,m-\ell+i}{Nu_1+m+i+1;,1}{1+s}.$$
 The result is obtained plugging the previous identity in \eqref{pnf} and comparing the resulting series expansion for $g_N$ with its definition.
\end{proof}
Now, we establish some consequences of the previous results.
\begin{corollary}[Mean]\label{mMM}
 The random variable $L_\infty^N$ has mean
 $$E\left[L_\infty^N\right]=\frac{N(s+u_0-u_1)+(1+Nu_1)p_1^N}{1+s+Nu_0}.$$
\end{corollary}
\begin{proof}
 Note that $E[L_\infty^N]=\lim_{z\rightarrow 1} g_N'(z)$. In addition, Eq. \eqref{edomm} yields
\begin{equation}\label{pp}
 \frac{(1+sz)\,g_N'(z)}{N}=sg_N(z)+u_0-u_1 \frac{g_N(z)}{z} -\frac{u_0(1-g_N(z))}{1-z}+\frac{\left(1+N u_1\right)p_1^N\,}{N}.
\end{equation}
The result follows by letting $z\to1$ in the previous identity.
\end{proof}
\begin{proposition}[Factorial moments]\label{fmMM}
The factorial moments of $L_\infty^N$ satisfy for $n\in[N-1]$
$$((n+1)(1+s)+Nu_0)\,E\left[(L_\infty^N)_{n+1}^{\downarrow}\right]=(n+1)(N-n)s\,E\left[(L_\infty^N)_{n}^{\downarrow}\right]- N(n+1)u_1\,E\left[(L_\infty^N-1)_n^{\downarrow}\right],$$
where $()^{\downarrow}$ is the falling factorial (see Appendix \ref{Ab}). Moreover, for $n\in[N]$
 \begin{enumerate}
  \item if $u_0=0$, then
   \begin{equation*}
E[(L_\infty^N)_n^{\downarrow}]= n!\left(\pFq{2}{1}{n+1;,n+1-N}{Nu+n+2}{-s}p_{n+1}^N+\pFq{2}{1}{n;,n-N}{Nu+n+1}{-s}p_{n}^N\right).
\end{equation*}
\item if $u_1=0$, then
$$E\left[(L_\infty^N)_{n}^{\downarrow}\right]=\frac{n!(N-1)_{n-1}^\downarrow}{\left(2+\frac{Nu}{1+s}\right)_{n-1}^{\uparrow}}\left(\frac{s}{1+s}\right)^{n-1}E[L_\infty^N].$$
\end{enumerate}
\end{proposition}
\begin{proof}
Differentiating Eq. \eqref{pp} $n$ times and using the general Leibniz rule we obtain
\begin{align}\label{dnp}
 \frac{(1+sz)\,g_N^{(n+1)}(z)}{N}+ \frac{ns\, g_N^{(n)}(z)}{N}&=sg_N^{(n)}(z)-u_1\sum\limits_{k=0}^n\binom{n}{k}\frac{g_N^{(n-k)}(z)(-1)^k k!}{z^{k+1}}\nonumber\\
 &-\frac{u_0 n!(-1)^n}{(z-1)^{n+1}}\left(\sum\limits_{k=0}^n\frac{g_N^{(k)}(z)(z-1)^k(-1)^{k}}{k!}-1\right).
\end{align}
Since $\lim_{z\rightarrow 1}g_N^{(n)}(z)=E[(L_\infty^N)_n^{\downarrow}]$, we have
\begin{equation}\label{l1}
 \lim\limits_{z\rightarrow 1}\sum\limits_{k=0}^n\binom{n}{k}\frac{g_N^{(n-k)}(z)(-1)^k k!}{z^{k+1}}=E\left[\sum\limits_{k=0}^n\binom{n}{k}(L_\infty^N)_{n-k}^{\downarrow}(-1)_k^{\downarrow}\right]=E\left[(L_\infty^N-1)_n^{\downarrow}\right].
\end{equation}
In addition, using l'H\^{o}pital's rule, we get for all $n\in\Nb$
\begin{align}\label{l2}
 \lim\limits_{z\rightarrow 1}\frac{1}{(z-1)^{n+1}}&\left(
 \sum\limits_{k=0}^n\frac{g_N^{(k)}(z)(z-1)^k(-1)^{k}}{k!}-1\right)\nonumber\\&=\lim\limits_{z\rightarrow 1}\frac{1}{(n+1)(z-1)^{n}}\sum\limits_{k=0}^n \left(\frac{g_N^{(k+1)}(z)(z-1)^k(-1)^{k}}{k!}-\frac{g_N^{(k)}(z) k(z-1)^{k-1}(-1)^{k-1}}{k!}\right)\nonumber\\
 &= \lim\limits_{z\rightarrow 1}\frac{g_N^{(n+1)}(z)(-1)^{n}}{n!(n+1)}=\frac{E\left[(L_\infty^N)_{n}^{\downarrow}\right](-1)^{n}}{(n+1)!}.
\end{align}
The first statement follows letting $z\to1$ in \eqref{dnp} and using \eqref{l1} and \eqref{l2}.

Now, we proceed to prove (1). First note that $E[(L_\infty^N)_n^{\downarrow}]=\sum_{j=n}^N (j)_n^{\downarrow}\, p_j^N$. Therefore, using \eqref{pmfMMu0} we get
$$E[(L_\infty^N)_n^{\downarrow}]=p_1^N\sum\limits_{j=n}^N\frac{(N-1)_{j-1}^{\downarrow}}{(Nu+2)_{j-1}^{\uparrow}}\,s^{j-1}(j)_n^{\downarrow}=p_1^N s^{n-1} f^{(n)}(s),$$
where $f(x)=\sum_{j=1}^N\frac{(N-1)_{j-1}^{\downarrow}}{(Nu+2)_{j-1}^{\uparrow}} x^j= \pFq{2}{1}{1;,1-N}{Nu+2}{-x}\, x$. The result follows from \cite[p.~241, Eq.~(9.2.3)]{leb}.

Assertion (2) follows directly iterating the first statement with $u_1=0$. 
\end{proof}

\section{\texorpdfstring{The master equation for the $\Lambda$-Wright--Fisher model}{The master equation for the Lambda-Wright--Fisher model}}\label{s5}
As in the previous section, we denote by $D:=\{z\in\Cb: |z|<1\}$ the open unit disk. In this section we aim to characterise the probability generating function $g_\Lambda:D\to\Cb$ defined via $g_\Lambda(z):=\sum_{k=1}^\infty p_k^\Lambda z^k$. Since $p_n^\Lambda=a_{n-1}^\Lambda-a_n^\Lambda$, Eq. \eqref{frlm} turns into
\begin{equation}\label{prlm}
\left(\frac{m_0(n+1)}{2}+\theta_1\right)p_{n+1}^\Lambda+\sum\limits_{k=n+1}^\infty p_k^\Lambda\left(c_{n,k}+\frac{m_1}{n}+\theta_0\right) =\sigma p_{n}^\Lambda, \quad n\in\Nb,
\end{equation}
where
$$c_{n,k}:=\frac{1}{n}\sum_{\ell=k+1}^\infty\binom{\ell-1}{\ell-n}\int_{(0,1)}\xi^{\ell-n-2}(1-\xi)^n\Lambda_0({\rm d}\xi),\quad k>n.$$
The recursion is completed with the condition $\sum_{k=1}^\infty p_k^\Lambda=1$.

For $z_1,z_2\in D$ and any analytic function $f:D\to\Cb$ we denote by $\int_{z_1}^{z_2} f(\xi){\rm d}\xi$ the integral of $f$ along any smooth path in $D$ connecting $z_1$ and $z_2$. 
\begin{proposition}[Master equation I]\label{MEI}
 For all $z\in D\setminus\{0\}$,
 \begin{align*}
  \frac{m_0}{2}g_\Lambda'(z)+m_1\int_0^z\frac{u-g_\Lambda(u)}{u(1-u)}{\rm d}u+\frac{\sigma z^2- (\sigma+\theta)z+\theta_1}{z(1-z)}\,g_\Lambda(z)&=\left(\frac{m_0}{2} +\theta_1 \right)p_1^\Lambda - \frac{\theta_0 z}{1-z}-\sum\limits_{k=2}^\infty p_k^\Lambda c_{k}(z),
  \end{align*}
where $c_k(z):=\sum_{n=1}^{k-1} c_{n,k}z^n$.
\end{proposition}
\begin{proof}
 Let $z\in D\setminus\{0\}$.  Multiplying \eqref{prlm} with $z^n$ and summing over all $n\in\Nb$ leads to
\begin{equation}\label{gf1}
\sum\limits_{n=1}^\infty \left(\frac{m_0(n+1)}{2}+\theta_1\right) p_{n+1}^\Lambda z^n+\sum\limits_{n=1}^\infty z^n\sum\limits_{k=n+1}^\infty p_k^\Lambda\left(c_{n,k}+\frac{m_1}{n}+\theta_0\right) =\sigma g_\Lambda(z).
\end{equation}
Note that $\sum_{n=1}^\infty n p_n^\Lambda z^{n-1}= g_\Lambda'(z)$. Therefore,  
\begin{equation}\label{lg1}
\sum\limits_{n=1}^\infty \left(\frac{m_0(n+1)}{2}+\theta_1\right) p_{n+1}^\Lambda z^n=\frac{m_0}2 \left(g_\Lambda'(z)-p_1^\Lambda\right)+\theta_1 \,\frac{(g_\Lambda(z)-z p_1^\Lambda)}{z}.
\end{equation}
In addition, using Fubini's theorem, we get 
\begin{equation}\label{r1}
\sum\limits_{n=1}^\infty z^n\sum\limits_{k=n+1}^\infty p_k^\Lambda\left(c_{n,k}+\theta_0\right)=\sum\limits_{k=2}^\infty p_k^\Lambda\left( c_k(z)+\frac{z-z^k}{1-z}\right)=\sum\limits_{k=2}^\infty p_k^\Lambda c_k(z)+\frac{z-g_\Lambda(z)}{1-z}.
\end{equation}
Similarly, we have
\begin{equation}\label{r2}
\sum\limits_{n=1}^\infty \frac{z^n}{n}\sum\limits_{k=n+1}^\infty p_k^\Lambda=\int_0^z \left( \sum\limits_{n=1}^\infty u^{n-1}\sum\limits_{k=n+1}^\infty p_k^\Lambda\right){\rm d}u=\int_0^z \frac{u-g_\Lambda(u)}{u(1-u)}{\rm d}u.
\end{equation}
Plugging \eqref{lg1}, \eqref{r1}, \eqref{r2} in \eqref{gf1} yields the result.
\end{proof}

 \begin{proposition}[Master equation II]\label{MEII}
 For all $z\in D\setminus\{0\}$,
  \begin{align*}
 & \frac{m_0}{2}g_\Lambda'(z)+m_1\int_0^z\frac{u-g_\Lambda(u)}{u(1-u)}{\rm d}u+\frac{\sigma z^2- (\sigma+\theta)z+\theta_1}{z(1-z)}\,g_\Lambda(z)=\left(\frac{m_0}{2} +\theta_1 \right)p_1^\Lambda - \frac{\theta_0 z}{1-z}\\
 &\qquad\qquad\qquad-\int_{(0,1)}\frac{\Lambda_0({\rm d}\xi)}{\xi^2}\left[\int_{z(1-\xi)}^z  \frac{u-g_\Lambda(u)}{u(1-u)}{\rm d}u-\int_{z(1-\xi)}^{\xi+z(1-\xi)}  \frac{1-g_\Lambda(u)}{1-u}{\rm d}u+\int_{0}^{\xi}\frac{1-g_\Lambda(u)}{1-u}{\rm d}u  \right].
  \end{align*}
  \end{proposition}
\begin{proof}
From Proposition \ref{MEI} it suffices to show that
$$\sum\limits_{k=2}^\infty p_k^\Lambda c_{k}(z)=\int_{(0,1)}\frac{\Lambda_0({\rm d}\xi)}{\xi^2}\left[\int_{z(1-\xi)}^z  \frac{u-g_\Lambda(u)}{u(1-u)}{\rm d}u-\int_{z(1-\xi)}^{\xi+z(1-\xi)}  \frac{1-g_\Lambda(u)}{1-u}{\rm d}u+\int_{0}^{\xi}\frac{1-g_\Lambda(u)}{1-u}{\rm d}u  \right].$$
Using Fubini's theorem, we deduce that
\begin{equation}\label{scnk}
\sum\limits_{k=2}^\infty p_k^\Lambda c_{k}(z)=\int_{(0,1)}\frac{\Lambda_0({\rm d}\xi)}{\xi^2}\left(\int_0^z\sum_{k=2}^\infty p_k^\Lambda\, C_k(u,\xi)\,{\rm d}u\right).
\end{equation}
where for $u\in D\setminus\{0\}$ and $\xi\in(0,1)$
$$C_k(u,\xi):=\sum_{n=0}^{k-2}u^{n}\sum_{\ell=k}^\infty\binom{\ell}{n}\xi^{\ell-n}(1-\xi)^{n+1}=(1-\xi)\sum_{n=0}^{k-2}\left(\frac{u(1-\xi)}{\xi}\right)^n\sum_{\ell=k}^\infty\binom{\ell}{n}\xi^{\ell}.$$
Since
\begin{align*}
 \sum_{\ell=k}^\infty\binom{\ell}{n}\xi^{\ell}&=\sum_{\ell=n}^\infty\binom{\ell}{n}\xi^{\ell}-\sum_{\ell=n}^{k-1}\binom{\ell}{n}\xi^{\ell}=\frac{\xi^n}{(1-\xi)^{n+1}}-\binom{k-1}{n}\xi^{k-1}-\sum_{\ell=n}^{k-2}\binom{\ell}{n}\xi^{\ell},
\end{align*}
we deduce that
\begin{align*}
C_k(u,\xi)&=\sum_{n=0}^{k-2}u^n-(1-\xi)\left[\xi^{k-1}\sum_{n=0}^{k-2}\binom{k-1}{n}\left(\frac{u(1-\xi)}{\xi}\right)^n+\sum_{n=0}^{k-2}\left(\frac{u(1-\xi)}{\xi}\right)^n\sum_{\ell=n}^{k-2}\binom{\ell}{n}\xi^{\ell}\right]\\
 &=\frac{1-u^{k-1}}{1-u}-(1-\xi)\left[\left(\xi+u(1-\xi)\right)^{k-1}-\left(u(1-\xi)\right)^{k-1}+\sum_{\ell=0}^{k-2}\xi^\ell\sum_{n=0}^{\ell}\binom{\ell}{n}\left(\frac{u(1-\xi)}{\xi}\right)^n\right]\\
 &=\frac{1-u^{k-1}}{1-u}-(1-\xi)\left[\left(\xi+u(1-\xi)\right)^{k-1}-\left(u(1-\xi)\right)^{k-1}+\sum_{\ell=0}^{k-2}\left(\xi+u(1-\xi)\right)^\ell\right]\\
  &=\frac{1-u^{k-1}}{1-u}-(1-\xi)\left[\frac{1-\left(\xi+u(1-\xi)\right)^{k}}{1-\xi-u(1-\xi)}-\left(u(1-\xi)\right)^{k-1}\right].
\end{align*}
As a consequence, we obtain
\begin{align*}
 \sum_{k=2}^\infty p_k^\Lambda\, C_k(u,\xi)=&\frac{u-g_\Lambda(u)}{u(1-u)}-(1-\xi)\left[\frac{1-g_\Lambda(\xi+u(1-\xi))}{1-\xi-u(1-\xi)}-\frac{g_\Lambda(u(1-\xi))}{u(1-\xi)}\right].
\end{align*}
Integrating over $u\in(0,z)$ and making appropriate change of variables, we get
\begin{align*}
\int_0^z  \sum_{k=2}^\infty p_k^\Lambda\, C_k(u,\xi){\rm d}u
&=\int_0^z  \frac{u-g_\Lambda(u)}{u(1-u)}{\rm d}u-\int_{\xi}^{\xi+z(1-\xi)}  \frac{1-g_\Lambda(v)}{1-v}{\rm d}v+\int_{0}^{z(1-\xi)}\frac{g_\Lambda(v)}{v}{\rm d}v\\
&=\int_{z(1-\xi)}^z  \frac{u-g_\Lambda(u)}{u(1-u)}{\rm d}u-\int_{\xi}^{\xi+z(1-\xi)}  \frac{1-g_\Lambda(v)}{1-v}{\rm d}v+\int_{0}^{z(1-\xi)}\frac{1-g_\Lambda(v)}{1-v}{\rm d}v\\
&=\int_{z(1-\xi)}^z  \frac{u-g_\Lambda(u)}{u(1-u)}{\rm d}u-\int_{z(1-\xi)}^{\xi+z(1-\xi)}  \frac{1-g_\Lambda(v)}{1-v}{\rm d}v+\int_{0}^{\xi}\frac{1-g_\Lambda(v)}{1-v}{\rm d}v.
\end{align*}
The result follows.
\end{proof}
As a first application of the results obtained in this section we rediscover the geometric law arising in the Crow--Kimura model.
\begin{corollary}[The Crow--Kimura model]
If $\Lambda\equiv 0$, and $\theta_0>0$ or $\theta_1>\sigma$, then for all $z\in D$
 \begin{equation*}
  g_\Lambda(z)=\frac{(1-p)z}{1-pz},
 \end{equation*}
where $p$ is given in \eqref{pargeo}. In particular, $L_\infty^{\Lambda}\sim\textrm{Geom}(1-p)$.
\end{corollary}
\begin{proof}
 In this case $\sigma_\Lambda=0$ and Lemma \ref{posrec} implies that, if $\theta_0>0$ or $\theta_1>\sigma$, the process $L^\Lambda$ is positive recurrent. Moreover, Proposition \ref{MEII} yields
  \begin{equation*}
  g_\Lambda(z)=\frac{z[\theta_1p_1^\Lambda (1-z)-\theta_0 z]}{\sigma z^2- (\sigma+\theta)z +\theta_1},\quad z\in D\setminus\{0\}.
\end{equation*}
The result for $\theta_1=0$ follows directly. For $\theta_1>0$, the map $z\mapsto \sigma z^2- (\sigma+\theta)z +\theta_1$ has exactly one root in $D$, which is given by $z_0:=(\sigma+\theta-\sqrt{(\sigma+\theta)^2-4\sigma \theta_1})/(2\sigma)$. Since $g_\Lambda$ is analytic in $D$, we conclude that $p_1^\Lambda= \theta_0 z_0/(\theta_1(1-z_0))$. Plugging this expression in the formula for $g_\Lambda$ yields the result.
\end{proof}
\begin{remark}
Note that the function $h_\Lambda$ encoding the ancestral type distribution in the $\Lambda$-Wright--Fisher model (see Remark \ref{catd}) is given by $h_\Lambda(x)=1-g_\Lambda(1-x)$, $x\in[0,1].$
\end{remark}
\section{Solving the Fearnhead recursion for the Wright--Fisher diffusion model}\label{s6}
In this section we assume that the measure $\Lambda$ is concentrated in $0$ with total mass $m_0:=\Lambda(\{0\})$, i.e. blocks merge according to the Kingman coalescent. In particular, $\sigma_\Lambda=\infty$, and therefore, the block counting process is positive recurrent for any $\sigma>0$. Note that \eqref{prlm} reads
\begin{equation}\label{prk}
\left(\frac{m_0(n+1)}{2}+\theta_1\right)p_{n+1}^\Lambda=\sigma p_{n}-\theta_0 \sum\limits_{k=n+1}^\infty p_k^\Lambda,\quad n\in\Nb.
\end{equation}
The boundary condition $a_0^\Lambda=1$ yields $\sum_{n=1}^\infty p_n^\Lambda=1$. The following is the main result of this section.
\begin{theorem}\label{mainWF}
For the Wright--Fisher diffusion model with selection parameter $\sigma>0$ and mutation parameters $\theta_0,\theta_1\geq 0$ the following holds
\begin{itemize}
 \item[(i)] If $\theta_0=0$, then 
 \begin{equation}\label{pmfWFu0}
p_n^\Lambda=\frac{1}{\pFq{1}{1}{1}{2+\frac{2\theta}{m_0}}{\frac{2\sigma}{m_0}}}\frac{\left(\frac{2\sigma}{m_0}\right)^{n-1}}{\left(2+\frac{2\theta}{m_0}\right)_{n-1}^{\uparrow}}, \quad n\in\Nb,
\end{equation}
where $\chf$ is the confluent hypergeometric function (see Appendix \ref{Ab}). In particular, we have
 \begin{equation}\label{pgfWFu0}
  g_\Lambda(z)=\frac{\pFq{1}{1}{1}{2+\frac{2\theta}{m_0}}{\frac{2\sigma z}{m_0}}}{\pFq{1}{1}{1}{2+\frac{2\theta}{m_0}}{\frac{2\sigma}{m_0}}},\quad z\in\Cb.
 \end{equation}
\item[(ii)] If $\theta_0>0$, then for all $z\in D$
\begin{equation}\label{pgfWF}
 g_\Lambda(z)=\frac{2\theta_0 \,I_0}{m_0(I_0-I_1)}e^{\frac{2\sigma }{m_0}\,z }z^{-\frac{2\theta_1}{m_0}}\,(1-z)^{-\frac{2\theta_0}{m_0}}\int_{0}^z\left(\frac{I_1}{I_0}-\xi\right)\,\xi^{\frac{2\theta_1}{m_0}}\,(1-\xi)^{\frac{2\theta_0}{m_0}-1} e^{-\frac{2\sigma}{m_0}\,\xi}\,{\rm d}\xi,
\end{equation}
where $I_i=\int_{0}^1 y^{\frac{2\theta_1}{m_0}+i}\,(1-y)^{\frac{2\theta_0}{m_0}-1}e^{-\frac{2\sigma}{m_0}\, y}\,{\rm d}y$, $i\in\{0,1\}$. Moreover,
\begin{equation}\label{pmfWF}
p_n^\Lambda=\frac{2\theta_0 }{(I_0-I_1)}\left[\frac{I_1}{2\theta_1+m_0}\,q_{n,1}-\frac{I_0}{2\theta_1+m_0}\,q_{n,2}\right],\quad n\in\Nb,
\end{equation}
 where $q_{1,1}:=1$, $q_{1,2}:=0$, and for $n\geq2$
\begin{align*}
q_{n,i}:=\sum\limits_{m=0}^{n-i}\frac{\left(\frac{2\sigma}{m_0}\right)^m}{\left(\frac{2\theta_1}{m_0}+i+1\right)_m^\uparrow}  \pFq{3}{2}{m+1;,1-\frac{2\theta_0}{m_0};,m-n+i}{\frac{2\theta_1}{m_0}+m+i+1;,1}{1}.
\end{align*}
\end{itemize}

\end{theorem}
\begin{remark}
In the case $\theta=0$, Eq. \eqref{pmfWFu0} implies that $L_\infty^\Lambda$ is a Poisson random variable with parameter $2\sigma/m_0$ conditioned to be strictly positive (see \cite{PP13}). This together with Proposition \ref{mdlwfm} permits to recover the classical result of Kimura \cite{Ki62}
$$P_x(X_\infty=1)=\frac{1-e^{-\frac{2\sigma}{m_0}x}}{1-e^{-\frac{2\sigma}{m_0}}},\quad x\in[0,1].$$
\end{remark}
\begin{remark}
 Note that Proposition \ref{MEII} yields 
  \begin{equation}\label{odeWF}
 \frac{m_0}{2} z(1-z)g_\Lambda'(z)+\left(\sigma z^2-(\sigma+\theta)z+\theta_1\right)g_\Lambda(z)=\left(\frac{m_0}{2} +\theta_1 \right)p_1^\Lambda z(1-z)-\theta_0 z^2, \quad z\in D_*.
 \end{equation}
 We can solve this ODE and show Theorem \ref{mainWF} following the proof of Theorem \ref{mainMM}. We provide here an alternative approach based on the results of Section \ref{s2} and the following lemma.
\end{remark}
\begin{lemma}\label{m2k}
If $s=\sigma/N$, $u_1=\theta_1/N$, $u_0=\theta_0/N$ and $m_0=2$  then
 $$L_\infty^N\xrightarrow[N\rightarrow\infty]{(d)}L_\infty^\Lambda.$$
\end{lemma}
\begin{proof}
It suffices to show that $a_n^N\to a_n^\Lambda$ as $N\to\infty$ for all $n\in\Nb_0$. We do this by induction on $n\in\Nb$. Since $a_0^N=1=a_0^\Lambda$, the assertion is true for $n=0$. The case $n=1$ follows from \cite[Lemma~3]{KHB13}. Assume that the assertion holds for all $k<n$. Then \eqref{frmm} implies that the limit of $a_n^N$ exists and is related to $a_{n-1}^\Lambda$ and $a_{n-2}^\Lambda$ via \eqref{frlm}. Therefore, $\lim_{N\rightarrow\infty}a_n^N=a_n^\Lambda$.
\end{proof}
\begin{proof}[Proof of Theorem \ref{mainWF}] 
{{\bf (i)}:} Identity \eqref{pmfWFu0} is obtained by iteration of \eqref{prk} and imposing $\sum_{n\in\Nb}p_n^\Lambda=1$. Formula \eqref{pgfWFu0} is a direct consequence of Eq. \eqref{pmfWFu0}.

{{\bf (ii)}:} Since $L_\infty^\Lambda(\sigma,\theta_0,\theta_1,m_0)$ is distributed as $L_\infty^\Lambda(2\sigma/m_0,2\theta_0/m_0, 2\theta_1/m_0,2)$, we assume without loss of generality that $m_0=2$. For the Moran model with parameters $s=\sigma/N$, $u_1=\theta_1/N$ and $u_0=\theta_0/N$ Theorem \ref{mainMM} yields
\begin{equation}\label{pn2pk}
 g_N(z)=\frac{\theta_0 \,I_0^N}{(I_0^N-I_1^N)}\frac{\left(1+\frac{\sigma z}{N}\right)^{N+\theta_1+ \frac{N\theta_0}{N+\sigma}}}{z^{\theta_1}\,(1-z)^{\frac{N\theta_0}{N+\sigma}}}\int_{0}^z\frac{(\frac{I_1^N}{I_0^N}-\xi)\,\xi^{\theta_1}\,(1-\xi)^{\frac{N\theta_0}{N+\sigma}-1}}{\left(1+\frac{\sigma z}{N}\right)^{N+1+\theta_1+ \frac{N\theta_0}{N+\sigma}}}\,{\rm d}\xi.
\end{equation}
Lemma \ref{m2k} implies that $g_N(z)\to g_\Lambda(z)$ as $N\to\infty$. In addition, by dominated convergence we get  
$$\left(N/\sigma\right)^{N+1+\theta_1+ \frac{N\theta_0}{N+\sigma}}I_i^N\xrightarrow[N\rightarrow\infty]{}I_i,\quad i\in\{0,1\}.$$
Hence, letting $N\to\infty$ in \eqref{pn2pk} and using dominated convergence yields \eqref{pgfWF}. Moreover, a straightforward calculation shows that $\lim_{N\to\infty}q_{n,i}^N=q_{n,i}$, $i\in\{0,1\}$. Thus, \eqref{pmfWF} follows by letting $N\to\infty$ in \eqref{pmfMM}.  
\end{proof}
\begin{proposition}\label{fmkm}
The random variable $L_\infty^\Lambda$ has mean
 $$E\left[L_\infty^\Lambda\right]=\frac{2(\sigma+\theta_0-\theta_1)+(m_0+2\theta_1)p_1}{m_0+2\theta_0}.$$
Moreover, $L_\infty^\Lambda$ has factorial moments of all orders and they satisfy
\begin{equation}\label{fmWF}
 ((n+1)m_0+2\theta_0)\,E\left[(L_\infty^\Lambda)_{n+1}^{\downarrow}\right]=2(n+1)\sigma\,E\left[(L_\infty^\Lambda)_{n}^{\downarrow}\right]- 2(n+1)\theta_1\,E\left[(L_\infty^\Lambda-1)_n^{\downarrow}\right],\quad n\in\Nb.
\end{equation}
In addition,
\begin{enumerate}
 \item if $\theta_0=0$, then
  \begin{equation*}
E[(L_\infty^\Lambda)_k^{\downarrow}]=k!\left( \pFq{1}{1}{k+1}{k+2+\frac{2\theta}{m_0}}{\frac{2\sigma}{m_0}}p_{k+1}^\Lambda+ \pFq{1}{1}{k}{k+1+\frac{2\theta}{m_0}}{\frac{2\sigma}{m_0}}p_{k}^\Lambda\right),\quad k\in\Nb.
\end{equation*}
\item if $\theta_1=0$, then
$$E\left[(L_\infty^\Lambda)_{n}^{\downarrow}\right]=\frac{n!}{\left(2+\frac{2\theta}{m_0}\right)_{n-1}^{\uparrow}}\left(\frac{2\sigma}{m_0}\right)^{n-1}E[L_\infty^\Lambda], \quad n\in\Nb.$$
 \end{enumerate}
\end{proposition}

\begin{proof}
Without loss of generality we assume that $m_0=2$. First note that \eqref{prk} implies that $p_n^\Lambda\leq \sigma^{n-1}/(2+\theta_1)_{n-1}^{\uparrow}$, $n\in\Nb$. Thus, $L_\infty^\Lambda$ admits moments of all orders. Similarly, using 
\eqref{prmm} with $s=\sigma/N$, $u_1=\theta_1/N$, $u_0=\theta_0/N$, we  get $ p_n^N\leq \sigma^{n-1}/(2+\theta_1)_{n-1}^{\uparrow}$. Thus, by dominated convergence and Lemma \ref{m2k} we conclude that
$$E[(L_\infty^N)_k^{\downarrow}]=\sum_{n=k}^N p_n^N (n)_k^{\downarrow}\xrightarrow[N\rightarrow\infty]{}\sum_{n=k}^\infty p_n^\Lambda\,(n)_k^{\downarrow}=E[(L_\infty^\Lambda)_k^{\downarrow}].$$
The formula for the mean of $L_\infty^\Lambda $ and the recursion \eqref{fmWF} follow by letting $N\to\infty $ in Corollary \ref{mMM} and Proposition \ref{fmMM}, respectively. Now, we prove assertion (1). Note that \eqref{pmfWFu0} yields
$$E[(L_\infty^\Lambda)_k^{\downarrow}]=p_1\sum\limits_{n=k}^N\frac{\left(\frac{2\sigma}{m_0}\right)^{n-1}}{\left(2+\frac{2\theta}{m_0}\right)_{n-1}^{\uparrow}}(n)_k^{\downarrow}=p_1^\Lambda\left(\frac{2\sigma}{m_0}\right)^{k-1} f^{(k)}\left(\frac{2\sigma}{m_0}\right),$$
where $f(x)=\sum_{n=1}^\infty x^n/\left(2+\frac{2\theta}{m_0}\right)_{n-1}^{\uparrow} = \pFq{1}{1}{1}{2+\frac{2\theta}{m_0}}{x}\, x$. The result follows from \cite[p.~261, Eq.~(9.9.5)]{leb}. Assertion (2) follows iterating \eqref{fmWF} with $\theta_1=0$.
\end{proof}
\section{Solving the Fearnhead-type recursion for the star-shaped model}\label{s7}
In this section we assume that the measure $\Lambda$ is concentrated in $1$ with total mass $m_1:=\Lambda(\{1\})$, i.e. blocks merge according to the star-shaped coalescent. In particular, $\sigma_\Lambda=\infty$, and therefore, the block counting process is positive recurrent for any $\sigma>0$. Note that \eqref{frlm} reads
\begin{equation}\label{arss}
\left(\frac{m_1}{n}+\theta+\sigma\right)a_n^\Lambda =\sigma a_{n-1}^\Lambda+\theta_1 a_{n+1}^\Lambda, \quad n\in\Nb.
\end{equation}
In addition, $a_0^\Lambda=\sum_{n=1}^\infty p_n^\Lambda=1$. The following is the main result of this section.
\begin{theorem}\label{mainSS}
For the star-shaped model with selection parameter $\sigma>0$ and mutation parameters $\theta_0,\theta_1\geq 0$ the following holds
\begin{itemize}
 \item[(i)] If $\theta_1=0$, then 
 \begin{equation}\label{pmfSSu0}
p_n^\Lambda=\left(\frac{n\theta_0+m_1}{n(\sigma+\theta_0)+m_1}\right)\frac{(n-1)!}{\left(1+\frac{m_1}{\sigma+\theta_0}\right)_{n-1}^{\uparrow}}\left(\frac{\sigma}{\sigma+\theta_0}\right)^{n-1}, \quad n\in\Nb,
\end{equation}
and
 \begin{equation}\label{pgfSSu0}
  g_\Lambda(z)=1-(1-z)\,\pFq{2}{1}{1;,1}{1+\frac{m_1}{\sigma+\theta_0}}{\frac{\sigma z}{\sigma+\theta_0}},\quad z\in D.
 \end{equation}
\item[(ii)] If $\theta_1>0$, then for all $z\in D\setminus\{x_-\}$
\begin{equation}\label{pgfSS}
 g_\Lambda(z)=z\left(1-\frac{\sigma(1-z)}{\sigma z^2-(\sigma+\theta)z+\theta_1}\left(\frac{1-\frac{z}{x_+}}{1-\frac{z}{x_-}}\right)^{\frac{m_1}{d}}\int_z^{x_-}\left(\frac{1-\frac{u}{x_-}}{1-\frac{u}{x_+}}\right)^{\frac{m_1}{d}}{\rm d}u\right),
\end{equation}
where $d:=\sqrt{(\sigma+\theta)^2-4\sigma\theta_1}$, $x_-:=(\sigma+\theta-d)/(2\sigma)\in(0,1)$ and $x_+:=(\sigma+\theta+d)/(2\sigma)>1$. In particular,
$$p_1^\Lambda=1-\frac{\sigma}{\theta_1}\int_0^{x_-}\left(\frac{1-\frac{u}{x_-}}{1-\frac{u}{x_+}}\right)^{\frac{m_1}{d}}{\rm d}u.$$
\end{itemize}
\end{theorem}
\begin{proof}
 {{\bf (i)}:} In this case, Eq. \eqref{arss} takes the form $(m_1/n+\theta+\sigma)a_n^\Lambda =\sigma a_{n-1}^\Lambda$, $n\in\Nb$, with obvious solution
 \begin{equation}\label{anss}
  a_n^\Lambda=\frac{n!}{\left(1+\frac{m_1}{\sigma+\theta_0}\right)_{n}^{\uparrow}}\left(\frac{\sigma}{\sigma+\theta_0}\right)^{n}, \quad n\in\Nb_0.
 \end{equation}
 Plugging this expression in $p_n^\Lambda=a_{n-1}^\Lambda-a_n^\Lambda$ yields \eqref{pmfSSu0}. Moreover, \eqref{pgfSSu0} follows from \eqref{anss} and the identity $\sum_{n=0}^\infty a_n^\Lambda z^n=(1-g_\Lambda(z))/(1-z)$. 
 
  {{\bf (ii)}:} First note that Proposition \ref{MEII} yields
  \begin{equation}\label{DESS}
 m_1z(1-z)\int_0^z\frac{u-g_\Lambda(u)}{u(1-u)}{\rm d}u+\left(\sigma z^2-(\sigma+\theta)z+\theta_1\right)g_\Lambda(z)=\theta_1 p_1^\Lambda z(1-z)-\theta_0 z^2, \quad z\in D\setminus\{0\}.
\end{equation}
Hence, the function $f:D\setminus\{0\}\to\Cb$ defined via $f(z):=(z-g_\Lambda(z))/(z(1-z))$ satisfies
  $$m_1 \int_0^z f(u){\rm d}u -(\sigma z^2 -(\sigma+\theta)z+\theta_1)f(z)=\theta_1 (p_1^\Lambda-1)+\sigma z.$$
 Differentiating this equation leads to the following first order differential equation
 \begin{equation}\label{aux}
  (m_1+\sigma+\theta-2\sigma z) f(z)-(\sigma z^2 -(\sigma+\theta)z+\theta_1)f^{\prime}(z)=\sigma.
 \end{equation}
 The solution of the homogeneous differential equation $(m_1+\sigma+\theta-2\sigma z) f_0(z)=(\sigma z^2 -(\sigma+\theta)z+\theta_1)f_0^\prime(z),$
 is, up to a multiplicative constant, given by 
 $$f_0(z)= \frac{1}{\sigma z^2-(\sigma+\theta)z+\theta_1}\left(\frac{1-\frac{z}{x_+}}{1-\frac{z}{x_-}}\right)^{\frac{m_1}{d}},\quad z\in D\setminus\{x_-\},$$
 where $d:=\sqrt{(\sigma+\theta)^2-4\sigma\theta_1}$, $x_-:=(\sigma+\theta-d)/(2\sigma)$ and $x_+:=(\sigma+\theta+d)/(2\sigma)$ ($x_-$ and $x_+$ are the roots of the polynomial $z\mapsto\sigma z^2 -(\sigma+\theta)z+\theta_1$). Therefore, the solution of the inhomogeneous differential equation \eqref{aux} is of the form
 $$f(z)=f_0(z)\left(C-\sigma \int_0^{z}\left(\frac{1-\frac{u}{x_-}}{1-\frac{u}{x_+}}\right)^{\frac{m_1}{d}}{\rm d}u\right), \quad z\in D\setminus\{x_-\}.$$
 Since $f_0$ has a singularity at $z=x_-$, but $f$ is analytic in $D\setminus\{0\}$, we get $C=\sigma\int_0^{x_-}\left(\frac{1-u/x_-}{1-u/x_+}\right)^{\frac{m_1}{d}}{\rm d}u$. Plugging this value of $C$ into the previous formula for $f$ yields 
 \begin{equation}\label{lastf}
  f(z)=\sigma f_0(z)\int_z^{x_-}\left(\frac{1-\frac{u}{x_-}}{1-\frac{u}{x_+}}\right)^{\frac{m_1}{d}}{\rm d}u,\quad z\in D\setminus\{x_-\}.
 \end{equation}
Since $g_\Lambda(z)=z(1-(1-z)f(z))$, \eqref{pgfSS} follows. Letting $z\to 0$ in \eqref{lastf} yields the expression for $p_1^\Lambda$. 
\end{proof}

\begin{remark}
Making the substitution $y=(x_{-}-u)/(x_-z)$ in \eqref{lastf} and applying \eqref{irf21} we obtain
$$f(z)=\frac{d}{(m_1+d)(x_+-x_-)}\left(\frac{x_+-z}{x_+-x_-}\right)^{\gamma-1} \pFq{2}{1}{\frac{m_1}{d};,\frac{m_1}{d}+1}{\frac{m_1}{d}+2}{\frac{z-x_-}{x_+-x_-}},$$
for $z\in B:=\{w\in D: |z-x_-|<x_+-x_-\}$. Moreover, from \cite[p.~247, Eqs.~(9.5.1)~and~(9.5.2)]{leb}, the previous identity translates into
$$f(z)=\frac{d}{(m_1+d)(x_+-x_-)}\pFq{2}{1}{2;,1}{\frac{m_1}{d}+2}{\frac{z-x_-}{x_+-x_-}}, \quad z\in B.$$
From this expression one can easily obtain the coefficients of the series expansion of $f$ around $x_-$. However, a series expansion for $f$ around $0$ using this formula is only possible if $2x_-<x_+$ (i.e. $(\sigma+\theta)^2> 9\theta_1 \sigma/2 $). In this case, using \cite[p.~241, Eq.~(9.2.3)]{leb} we deduce that $f(z)=\sum_{k=0}^\infty f_k z^k$, where
$$f_k=\frac{d}{(m_1+d)}\frac{(2)_k^{\uparrow}\,(x_+-x_-)^{-k+1}}{(\frac{m_1}{d}+2)_k^{\uparrow}}\pFq{2}{1}{2+k;,1+k}{\frac{m_1}{d}+2+k}{\frac{-x_-}{x_+-x_-}}.$$
The coefficients $(p_k^\Lambda)_{k\in\Nb}$ are obtained by setting $p_1^\Lambda=1-f_0$ and $p_{k+1}^\Lambda=f_{k-1}-f_k$, $k\in\Nb$.

In the case, where $2x_-\geq x_+ $, we can proceed as follows. We set $a_1^\Lambda=1-p_1^\Lambda$. Then using $a_0^\Lambda=0$, we obtain the values $a_2^\Lambda, a_3^\Lambda,...$ by successive substitution in \eqref{arss}. Finally we set $p_n^\Lambda=a_{n-1}^\Lambda-a_n^\Lambda$.
\end{remark}

\section{Some further comments and results for the Bolthausen--Sznitman model}\label{s8}
\subsection{Solving the Fearnhead-type recursion for the Bolthausen--Sznitman model}
Let us assume that $\Lambda$ is the uniform measure on $[0,1]$, i.e. blocks merge according to the Bolthausen--Sznitman coalescent. Since in this case $\sigma_\Lambda=\infty$, then $L^\Lambda$ is positive recurrent for any $\sigma>0$. Moreover, we have shown in Section \ref{s2} that $L_\infty^\Lambda\sim\textrm{Geom}(1-\rho)$, where $\rho$ is the unique solution to Eq. \eqref{cg3b} (see Corollary \ref{geomlawBS}). In this section we would like to relate this result with the results obtained in Section \ref{s5}.
\begin{lemma}[Carleman integral equation]\label{MEBS}
The function $\rho_\Lambda$ defined via $\rho_\Lambda(x):=g_\Lambda(x)/x$ , $x\in (0,1)$, is a solution of the Carleman singular integral equation
\begin{equation}\label{CE}
 \alpha(x) \rho_\Lambda(x)-\vp_0^1 \frac{\rho_\Lambda(t)}{t-x}{\rm d}t=f(x),\quad x\in(0,1),
\end{equation}
where $\alpha(x):=\sigma+\log(1-x)-\log(x)-\frac{\theta_1}{x}+\frac{\theta_0}{1-x}$, $f(x):=\frac{\theta_0}{1-x}-\frac{\theta_1 p_1^\Lambda}{x}$ and $\vp_a^b h(t){\rm d}t$ denotes the Cauchy principal value of a function $h$ (provided this value exists).
\end{lemma}
\begin{proof}
 Since $c_{n,k}=1/(k-n)$, we have 
 $$c_k(x)=\sum\limits_{n=1}^{k-1}\frac{x^n}{k-n}=\int_0^1 u^{k-1}\sum_{n=1}^{k-1}\left(\frac{x}{u}\right)^n {\rm d}u=x\int_0^1\frac{u^{k-1}-x^{k-1}}{u-z}{\rm d}u,$$
and hence
 $$\sum_{k=2}^{\infty}p_k^\Lambda c_k(x)=x\int_0^1\frac{\rho_\Lambda(u)-\rho_\Lambda(x)}{u-x}{\rm d}u.$$
Combining this with Proposition \ref{MEI} we obtain 
\begin{equation}\label{ieBS}
\left(\frac{\theta_1}{x}-\frac{\theta_0}{1-x}-\sigma \right) \rho_\Lambda(x)+ \int_0^1\frac{\rho_\Lambda(x)-\rho_\Lambda(t)}{x-t}{\rm d}t=\frac{\theta_1 p_1^\Lambda}{x}-\frac{\theta_0}{1-x},\quad x\in(0,1).
 \end{equation}
The result follows using that $\int_0^1\frac{\rho_\Lambda(x)-\rho_\Lambda(t)}{x-t}{\rm d}t=\vp_0^1 \frac{\rho_\Lambda(t)}{t-x}{\rm d}t-\rho_\Lambda(x)\left(\log(1-x)-\log(x)\right)$.
\end{proof}
The solution of \eqref{CE} with boundary conditions $\lim_{x\to 0}\rho_\Lambda(x)=p_1^\Lambda$ and $\lim_{x\to 0}\rho_\Lambda(x)=1$ can be derived via the method described in \cite[Eq.~(2.1)]{EK87} (see also \cite[Sect.~4.4]{Tri}). This approach involves quite technical calculations and leads to rather complicated formulas for $\rho_\Lambda$ and $p_1^\Lambda$, from which it seems not straightforward to infer that the underlying distribution is geometric. However, knowing that $L_\infty^\Lambda\sim\textrm{Geom}(1-\rho)$ for some $\rho\in(0,1)$, we can deduce the value of $\rho$ from Lemma \ref{MEBS}. Indeed, in this case
  $$p_1^\Lambda=1-\rho\quad\textrm{and}\quad \rho_\Lambda(x)=\frac{1-\rho}{1-\rho x},\quad x\in(0,1).$$
Moreover, one can check that
$$\vp_0^1 \frac{\rho_\Lambda(t)}{t-x}{\rm d}t=\rho_{\Lambda}(x)(\log(1-x)-\log(x)-\log(1-\rho)),$$
and therefore
$$\alpha(x) \rho_\Lambda(x)-\vp_0^1 \frac{\rho_\Lambda(t)}{t-x}{\rm d}t=\rho_\Lambda(x)\left(\sigma-\frac{\theta_1}{x}+\frac{\theta_0}{1-x}+\log(1-\rho)\right).$$
In addition, we have
$$f(x)=\rho_{\Lambda}(x)\left(\frac{\theta_0}{1-x}+\frac{\theta_0\rho}{1-\rho}-\frac{\theta_1}{x}+\theta_1\rho\right).$$
Since $\rho_\Lambda$ satisfies \eqref{CE}, we infer that $(\sigma +\log(1-\rho)-\theta_1\rho)(1-\rho)+\theta_0\rho=0$, and therefore $\rho$ is the unique solution to Eq. \eqref{cg3b} (see Lemma \ref{uniquerho}).
\subsection{Solving the moments of the stationary distribution for the Bolthausen--Sznitman model}
Let us assume that $\theta_0,\theta_1>0$. In this section we aim to obtain an explicit expression for the generating function of the coefficients $w_n:=P_n(R_\infty=0)=E[(1-X_\infty)^n]$, $n\in\Nb_0$, for the Bolthausen--Sznitman model. Note that $w_0=1$. Moreover, in this case
\[
\binom{n}{n-\ell+1}\lambda_{n,n-\ell+1}\ =\ \frac{n}{(n-\ell)(n-\ell+1)},
\qquad \ell\in[n-1].
\]
Thus, the characteristic equations \eqref{momrec} for the sequence $(w_n)_{n\ge 0}$ take the form
\begin{equation} \label{wnchar}
   \bigg(\theta+\sigma+1-\frac{1}{n}\bigg)w_n
   \ =\ \theta_1w_{n-1} + \sigma w_{n+1}
   + \sum_{\ell=1}^{n-1} \frac{w_\ell}{(n-\ell)(n-\ell+1)},\quad n\in\Nb.
\end{equation}
In order to solve these characteristic equations we introduce the generating function $w(s):=\sum_{n\ge 1}w_ns^n$. From $w_n\in [0,1]$ we conclude that the function $w$ has at least radius of convergence $1$. In the following we use a particular convolution property of the Bolthausen--Sznitman model. The same factorization property has been successfully used to determine the so-called hitting probabilities for the Bolthausen--Sznitman coalescent \cite{M14} and more generally (see \cite[Eq.~(4.6)]{M14b}) for the $\beta(2-\alpha,\alpha)$-coalescent with parameter $0<\alpha<2$.

Multiplying the last sum in (\ref{wnchar}) with $s^n$ and summing over all $n\ge 2$ leads to the factorization
\begin{eqnarray*}
   \sum_{n= 2}^\infty s^n\sum_{\ell=1}^{n-1}\frac{w_\ell}{(n-\ell)(n-\ell+1)}
   & = & \sum_{\ell= 1}^\infty w_\ell s^\ell \sum_{n=\ell+1}^\infty \frac{s^{n-\ell}}{(n-\ell)(n-\ell+1)}\\
   & = & \sum_{\ell=1}^\infty w_\ell s^{\ell} \sum_{k=1}^\infty\frac{s^k}{k(k+1)}
   \ = \ w(s)\varphi(s),
\end{eqnarray*}
where $\varphi$ is the probability generating function of a random variable $\eta$ with distribution $p_k:=P(\eta=k)=1/(k(k+1))$, $k\in\Nb$, i.e.
\[
\varphi(s)\ :=\ \sum_{k= 1}^\infty\frac{s^k}{k(k+1)}\ =\ 1 + \frac{(1-s)\log(1-s)}{s}.
\]
The occurrence of the generating function $\varphi$ is typical for the
Bolthausen--Sznitman model (see, for example, Drmota et al. \cite[p.~1409]{DIMR07} or H\'enard \cite[p.~3016]{He15} and comes from the fact that the (Siegmund dual) fixation line is a continuous-time branching process with offspring distribution $(p_k)_{k\ge 1}$.
Thus, multiplying (\ref{wnchar}) with $s^n$ and summing over all $n\in\Nb$ leads to
\[
   (\theta+\sigma+1)w(s) - \int_0^s \frac{w(t)}{t}\,{\rm d}t\\
   \ =\ \theta_1s(w(s)+1) + \sigma\bigg(\frac{w(s)}{s}-w_1\bigg)
+ \varphi(s)w(s).
\]
Taking the derivative with respect to $s$ shows that the generating function $w$ satisfies the inhomogeneous first order differential equation
\begin{equation*}
(\theta+\sigma+1)w'(s) - \frac{w(s)}{s}= \theta_1\big(w(s)+sw'(s)+1\big) + \sigma\bigg(\frac{w'(s)}{s}-\frac{w(s)}{s^2}\bigg)
   + \varphi'(s)w(s) + \varphi(s)w'(s).
\end{equation*}
Resorting leads to
\[
\bigg(\theta+\sigma+1-\theta_1s -\frac{\sigma}{s}-\varphi(s)\bigg)w'(s)\ =\ \bigg(\frac{1}{s}+\theta_1-\frac{\sigma}{s^2}+\varphi'(s)\bigg)w(s) + \theta_1.
\]
Plugging in $\varphi(s)=1+\frac{1-s}{s}\log(1-s)$ and $\varphi'(s)=-\frac{1}{s}-\frac{\log(1-s)}{s^2}$ yields
\[
   \bigg(
            \theta-\theta_1s-\sigma\frac{1-s}{s}-\frac{(1-s)\log(1-s)}{s}
         \bigg)w'(s)
   \ = \ \bigg(\theta_1-\frac{\sigma+\log(1-s)}{s^2}\bigg)w(s) + \theta_1
\]
or, in standard form,
\begin{equation} \label{dgl}
   w'(s)\ =\ a(s)w(s)+b(s),
\end{equation}
where
\[
a(s)\ :=\ \frac{\theta_1s^2-\sigma-\log(1-s)}{s\big(\theta s-\theta_1s^2-\sigma
(1-s)-(1-s)\log(1-s)\big)}
\]
and
\[
b(s)\ :=\ \frac{\theta_1s}{\theta s-\theta_1s^2-\sigma(1-s)-(1-s)\log(1-s)}.
\]
The function $a(\cdot)$ has a singularity
at $s_1:=0$ and at another point $s_2\in (0,1)$ being the root of the map
$h(s):=\theta s-\theta_1s^2-\sigma(1-s)-(1-s)\log(1-s)$ satisfying
$h(0)=-\sigma<0$ and $h(1)=\theta-\theta_1>0$. We can therefore choose some
fixed $s_0\in (0,s_2)$ and write a particular solution $w_0(\cdot)$ of the homogeneous differential equation $w'(s)=a(s)w(s)$ in the form
\[
w_0(s)\ =\ \exp\bigg(\int_{s_0}^s a(t){\rm d}t\bigg),\qquad s\in (0,s_2).
\]
The solution of the differential equation (\ref{dgl})
with initial value $w(0)=0$ is hence
\[
w(s)\ =\ w_0(s)\int_0^s \frac{b(t)}{w_0(t)}{\rm d}t,\qquad s\in (0,s_2).
\]
\begin{remark}Note that the Stieltjes's transform of the law of $1-X_\infty$ is expressed in terms of $w$ as 
$$\Ss(t):=E\left[\frac{1}{t-(1-X_\infty)}\right]=\frac{w\left(\frac{1}{t}\right)-1}{t},\quad t>1/s_2.$$
\end{remark}
\section{\texorpdfstring{A remark on the $\beta(3,1)$-model}{A remark on the beta(3,1)-model }}\label{s9}
There is another instance where the general method in Section \ref{s5} leads to a simple ordinary differential equation. This is given by the $\beta(3,1)$-model, i.e. the $\Lambda$-Wright--Fisher model with $\Lambda({\rm d}x)=3x^2 {\rm d}x$. Indeed, in this case
$c_{n,k}=3/(k+1)$, and hence $c_k(z)=\frac{3}{k+1}\frac{z-z^k}{1-z}$. Therefore,
$$\sum\limits_{k=2}^\infty p_k^\Lambda c_k(z)=\frac{3}{1-z}\left[z\sum\limits_{k=2}^\infty\frac{p_k^\Lambda}{k+1}-\sum\limits_{k=2}^\infty\frac{p_k^\Lambda}{k+1}z^k\right]=\frac{3}{1-z}\left[z\sum\limits_{k=1}^\infty\frac{p_k^\Lambda}{k+1}-\frac{1}{z}\int_0^z g_\Lambda(u){\rm d}u\right].$$
In addition, using Eq. \eqref{prlm} for $n=1$ we get
$$3\sum\limits_{k=1}^\infty\frac{p_k^\Lambda}{k+1}=\left(\frac{3}{2}+ \sigma +\theta_0\right)p_1^\Lambda-\theta_1 p_2^\Lambda-\theta_0.$$
Thus, Proposition \ref{MEI} yields
$$3\int_0^z g_\Lambda(u){\rm d}u-(\sigma z^2-(\sigma+\theta)z+\theta_1)g_\Lambda(z)=\left[\left(\left(\frac32+\sigma+\theta\right)p_1^\Lambda -\theta_1 p_2^\Lambda\right)z-\theta_1p_1^\Lambda\right]z.$$
Differentiating this equation, we deduce that $g_\Lambda$ solves the ordinary differential equation
$$(\sigma z^2-(\sigma+\theta)z+\theta_1)g_\Lambda'(z)+(2\sigma z-\sigma-\theta-3)g_\Lambda(z)=\theta_1 p_1^\Lambda-((3+2(\sigma+\theta))p_1^\Lambda-2\theta_1 p_2^\Lambda)z, \quad z\in D_*.$$
Explicit formulas for $g_\Lambda$ and $p_1^\Lambda$ can be obtained solving this equation with the boundary conditions $g_\Lambda(0)=0$ and $g_\Lambda(1)=1$. We leave the details to the reader.
\appendix
\section{Some special functions}\label{Ab}
The \textit{rising and falling factorials} $()^\uparrow$ and $()^\downarrow$ are defined as
$$(\alpha)_n^\uparrow:=\alpha(\alpha+1)\cdots(\alpha+n-1)\quad\textrm{and}\quad (\alpha)_n^\downarrow:=\alpha(\alpha-1)\cdots(\alpha-n+1),\quad n\in \Nb,$$
and $(\alpha)_0^\uparrow:=1=:(\alpha)_0^\downarrow$. 
The \textit{Gauss hypergeometric function} $\hf$ is the absolutely convergent power series
\begin{equation}\label{def21}
\pFq{2}{1}{\alpha;,\beta}{\gamma}{z}:=\sum\limits_{k=0}^\infty \frac{(\alpha)_k^\uparrow (\beta)_k^\uparrow}{(\gamma)_k^\uparrow}\frac{z^k}{k!},\qquad z\in D,
\end{equation}
where $\alpha,\beta,\gamma$ are parameters which can take real or complex values (provided that $\gamma\notin -\Nb_0$). The function $\hf$ admits the integral representation (see \cite[p.~239, Eq.~(9.1.4)]{leb})
\begin{equation}\label{irf21}
 \pFq{2}{1}{\alpha;,\beta}{\gamma}{z}=\frac{\Gamma(\gamma)}{\Gamma(\beta)\Gamma(\gamma-\beta)}\int_0^1 t^{\beta-1} (1-t)^{\gamma-\beta-1} (1-zt)^{-\alpha}{\rm d}t,\qquad \textrm{Re}(\gamma)> \textrm{Re}(\beta)>0.
\end{equation}
The \textit{confluent hypergeometric function} $\chf$ is the absolutely convergent power series
\begin{equation}\label{def11}
\pFq{1}{1}{\alpha}{\gamma}{z}:=\sum\limits_{k=0}^\infty \frac{(\alpha)_k^\uparrow}{(\gamma)_k^\uparrow}\frac{z^k}{k!},\qquad z\in \Cb,
\end{equation}
where $\alpha,\gamma$ are parameters which can take real or complex values (provided that $\gamma\notin -\Nb_0$). The function $\chf$ admits the integral representation (see \cite[p.~266, Eq.~(9.11.1)]{leb}).
\begin{equation}\label{irf11}
 \pFq{1}{1}{\alpha,}{\gamma}{z}=\frac{\Gamma(\gamma)}{\Gamma(\alpha)\Gamma(\gamma-\alpha)}\int_0^1 e^{tz}t^{\alpha-1} (1-t)^{\gamma-\alpha-1}{\rm d}t,,\qquad \textrm{Re}(\gamma)> \textrm{Re}(\alpha)>0.
\end{equation}
Similarly, the \textit{generalised hypergeometric function} $\HF$ is the power series
\begin{equation}\label{def31}
\pFq{3}{2}{\alpha;,\beta;, \gamma}{\delta;,\rho}{z}:=\sum\limits_{k=0}^\infty \frac{(\alpha)_k^\uparrow (\beta)_k^\uparrow (\gamma)_k^\uparrow}{(\delta)_k^\uparrow (\rho)_k^\uparrow}\frac{z^k}{k!},\qquad z\in D,
\end{equation}
where $\delta,\rho\notin -\Nb_0$. The functions $\hf$ and $\HF$ can be defined outside the disk $D$ by using analytic continuation. Moreover, when $\alpha$ or $\beta$ are nonpositive integers, $\hf$ reduces to a polynomial, and therefore, is well defined in the whole complex plane. The same holds for $\HF$ when $\alpha,\beta$ or $\gamma$ are nonpositive integers.

A natural two variables generalisation of the Gauss hypergeometric function is given by the \textit{Appell function} $F_1$ (see \cite{A82}), which is given by
\begin{equation}\label{def1}
\pFq{}{1}{a\,;,b\,;,c}{d}{z\,;,w}:=\sum\limits_{m=0}^\infty\sum\limits_{n=0}^\infty \frac{(a)_{m+n}^\uparrow (b)_m^{\uparrow} (c)_n^\uparrow}{(d)_{m+n}^\uparrow} \frac{z^m}{m!}\frac{w^n}{n!},\qquad z,w\in D,
\end{equation}
where $d$ is a non-positive integer. There are four types of Appell functions, but we focus here only on $F_1$. The function $F_1$ can be expressed in terms of $\hf$ functions as follows
$$\pFq{}{1}{a;,b;,c}{d}{z;,w}=\sum\limits_{m=0}^\infty\frac{(a)_{m}^\uparrow (b)_m^{\uparrow}}{(d)_{m}^\uparrow}\frac{z^m}{m!}\pFq{2}{1}{a+m;,c}{d+m}{w}.$$
The function $F_1$ admits the integral representation (see \cite[Eq.~(24)]{A82})
\begin{equation}\label{irf1}
 \pFq{}{1}{a;,b;,c}{d}{z;,w}=\frac{\Gamma(d)}{\Gamma(a)\Gamma(d-a)}\int_0^1 t^{a-1} (1-t)^{d-a-1} (1-zt)^{-b}(1-wt)^{-c}{\rm d}t,\quad\textrm{Re$(d)>$ Re$(a)>0$}.
\end{equation}
\section{Some integral identities}\label{Aa}
For $\alpha,\beta,\gamma,\nu>0$, we define $I(\alpha,\beta,\gamma,\nu;z):=\int_0^z y^\alpha(1-y)^\beta {(y+\nu)}^{-\gamma}{\rm d}y$, $z\in \Cb\setminus\Rb_-.$

\begin{lemma}\label{lA1} For every $z\in\Cb\setminus\Rb_-$, we have
$$I(\alpha,\beta,\gamma,\nu;z)= \nu^{\alpha-\gamma+1}\left(\frac{z}{z+\nu}\right)^{1+\alpha}\int_0^1 t^\alpha\left(1-\frac{z}{z+\nu}\,t\right)^{\gamma-\alpha-\beta-2}\left(1-\frac{(1+\nu)z}{z+\nu}\,t\right)^\beta {\rm d}t.$$ 
\end{lemma}
\begin{proof}
This follows directly by making the change of variable $t=(z+\nu)y/(z(y+\nu))$.
\end{proof}
\begin{corollary}\label{IvsF21}
We have
$$I(\alpha,\beta,\gamma,\nu;1)=\frac{\nu^{1+\alpha-\gamma}}{(1+\nu)^{1+\alpha}}\frac{\Gamma(1+\alpha)\Gamma(1+\beta)}{\Gamma(2+\alpha+\beta)}\,\pFq{2}{1}{2+\alpha+\beta-\gamma;,1+\alpha}{2+\alpha+\beta}{\frac{1}{1+\nu}  }.$$
\end{corollary}
\begin{proof}
This follows directly from Lemma \ref{lA1} and  Eq. \eqref{irf21}.
\end{proof}
Let $D_\nu:=\{z\in D_*: |z|<\nu/\sqrt{\nu^2+2\nu}\}.$ One can easily check that for all $z\in D_\nu$, $(1+\nu)z/(z+\nu)\in D$. 
\begin{corollary}\label{IvsF1}
 For all $z\in D_\nu$, we have
 $$I(\alpha,\beta,\gamma,\nu;z)= \frac{\nu^{\alpha-\gamma+1}}{1+\alpha}\left(\frac{z}{z+\nu}\right)^{1+\alpha}\,\pFq{}{1}{1+\alpha\,;,2+\alpha+\beta-\gamma\,;,-\beta}{2+\alpha}{\frac{z}{z+\nu}\,;,\frac{(1+\nu)z}{z+\nu}}.$$
\end{corollary}
\begin{proof}
This follows directly from Lemma \ref{lA1} and Eq. \eqref{irf1}.
\end{proof}
\footnotesize{{\bf{Acknowledgments}} The first author gratefully acknowledges financial support from the Deutsche Forschungsgemeinschaft via CRC 1283 Taming Uncertainty, Project C1.}


\bibliographystyle{abbrv}
\bibliography{reference}

\begin{thebibliography}{10}

\bibitem{A82}
P.~Appell.
\newblock Sur les fonctions hyperg\'{e}om\'{e}triques de deux variables.
\newblock {\em J. Math. Pures Appl.}, 8:173--216, 1882.

\bibitem{BCH17}
E.~Baake, F.~Cordero, and S.~Hummel.
\newblock A probabilistic view on the deterministic mutation-selection
  equation: dynamics, equilibria, and ancestry via individual lines of descent.
\newblock {\em J. Math. Biol.}, 77(3):795--820, 2018.

\bibitem{BLW16}
E.~Baake, U.~Lenz, and A.~Wakolbinger.
\newblock The common ancestor type distribution of a
  {$\Lambda$}-{W}right--{F}isher process with selection and mutation.
\newblock {\em Electron. Commun. Probab.}, 21(59):1--16, 2016.

\bibitem{BW18}
E.~Baake and A.~Wakolbinger.
\newblock Lines of descent under selection.
\newblock {\em Journal of Statistical Physics}, 172(1):156--174, 2018.

\bibitem{B09}
N.~Berestycki.
\newblock {\em Recent progress in coalescent theory}, volume~16 of {\em Ensaios
  Matem\'{a}ticos}.
\newblock Sociedade Brasileira de Matem\'{a}tica, Rio de Janeiro, 2009.

\bibitem{BLG03}
J.~Bertoin and J.-F. Le~Gall.
\newblock Stochastic flows associated to coalescent processes.
\newblock {\em Probab. Theory Related Fields}, 126(2):261--288, 2003.

\bibitem{BB09}
M.~Birkner and J.~Blath.
\newblock Measure-valued diffusions, coalescents and genetic inference.
\newblock In J.~Blath, P.~M\"orters, and M.~Scheutzow, editors, {\em Trends in
  Stochastic Analysis}, pages 329--364, Cambridge, 2009. Cambridge University
  Press.

\bibitem{BBGW18}
J.~Blath, E.~Buzzoni, A.~Gonz\'alez~Casanova, and M.~Wilke-Berenguer.
\newblock Structural properties of the seed bank and the two-island diffusion.
\newblock {\em preprint}, 2018.

\bibitem{BS98}
E.~Bolthausen and A.-S. Sznitman.
\newblock On {R}uelle's probability cascades and an abstract cavity method.
\newblock {\em Comm. Math. Phys.}, 197(2):247--276, 1998.

\bibitem{FC17}
F.~Cordero.
\newblock Common ancestor type distribution: a {M}oran model and its
  deterministic limit.
\newblock {\em Stochastic Process. Appl.}, 127(2):590--621, 2017.

\bibitem{FC17b}
F.~Cordero.
\newblock The deterministic limit of the {M}oran model: a uniform central limit
  theorem.
\newblock {\em Markov Process. Related Fields}, 23(2):313--324, 2017.

\bibitem{CGHJK96}
R.~M. Corless, G.~H. Gonnet, D.~E.~G. Hare, D.~J. Jeffrey, and D.~E. Knuth.
\newblock On the {L}ambert{W} function.
\newblock {\em Advances in Computational Mathematics}, 5(1):329--359, 1996.

\bibitem{CK56}
J.~F. Crow and M.~Kimura.
\newblock Some genetic problems in natural populations.
\newblock {\em Proc. Third Berkeley Symp. on Math. Statist. and Prob.},
  4:1--22, 1956.

\bibitem{DEP11}
R.~Der, C.~Epstein, and J.~B. Plotkin.
\newblock Generalized population models and the nature of genetic drift.
\newblock {\em Theor. Popul. Biol.}, 80(2):80--99, 2011.

\bibitem{DEP12}
R.~Der, C.~Epstein, and J.~B. Plotkin.
\newblock Dynamics of neutral and selected alleles when the offspring
  distribution is skewed.
\newblock {\em Genetics}, 191(4):1331--1344, 2012.

\bibitem{DK99}
P.~Donnelly and T.~G. Kurtz.
\newblock Particle representations for measure-valued population models.
\newblock {\em Ann. Probab.}, 27(1):166--205, 1999.

\bibitem{DIMR07}
M.~Drmota, A.~Iksanov, M.~M\"ohle, and U.~Roesler.
\newblock Asymptotic results concerning the total branch length of the
  {B}olthausen-{S}znitman coalescent.
\newblock {\em Stochastic Process. Appl.}, 117:1404--1421, 2007.

\bibitem{D08}
R.~Durrett.
\newblock {\em Probability {M}odels for {DNA} {S}equence {E}volution}.
\newblock Springer, New York, 2nd edition, 2008.

\bibitem{EK87}
R.~Estrada and R.~P. Kanwal.
\newblock The {C}arleman type singular integral equations.
\newblock {\em SIAM Review}, 29(2):263--290, 1987.

\bibitem{E11}
A.~Etheridge.
\newblock {\em Some {M}athematical {M}odels from {P}opulation {G}enetics}.
\newblock Springer, Heidelberg, 2011.

\bibitem{EG09}
A.~M. Etheridge and R.~C. Griffiths.
\newblock A coalescent dual process in a {M}oran model with genic selection.
\newblock {\em Theor. Popul. Biol.}, 75(4):320--330, 2009.

\bibitem{EGT10}
A.~M. Etheridge, R.~C. Griffiths, and J.~E. Taylor.
\newblock A coalescent dual process in a {M}oran model with genic selection,
  and the lambda coalescent limit.
\newblock {\em Theor. Popul. Biol.}, 78(2):77f--92, 2010.

\bibitem{EK86}
S.~N. Ethier and T.~G. Kurtz.
\newblock {\em Markov {P}rocesses: {C}haracterization and {C}onvergence}.
\newblock Wiley, New York, 1986.

\bibitem{Fe02}
P.~Fearnhead.
\newblock The common ancestor at a nonneutral locus.
\newblock {\em J. Appl. Probab.}, 39(1):38--54, 2002.

\bibitem{Fo13}
C.~Foucart.
\newblock The impact of selection in the {$\Lambda$}-{W}right--{F}isher model.
\newblock {\em Electron. Commun. Probab.}, 18(72):1--10, 2013.

\bibitem{GM16}
F.~Gaiser and M.~M\"ohle.
\newblock On the block counting process and the fixation line of exchangeable
  coalescents.
\newblock {\em ALEA Lat. Am. J. Probab. Math. Stat.}, 13(2):809--833, 2016.

\bibitem{GS18}
A.~Gonz\'alez~Casanova and D.~Span\`o.
\newblock Duality and fixation in {$\Xi$}-{W}right--{F}isher processes with
  frequency-dependent selection.
\newblock {\em Ann. Appl. Probab.}, 28(1):250--284, 2018.

\bibitem{G14}
R.~C. Griffiths.
\newblock The {$\Lambda$}-{F}leming--{V}iot process and a connection with
  {W}right--{F}isher diffusion.
\newblock {\em Adv. Appl. Probab.}, 46(4):1009--1035, 2014.

\bibitem{Ha21}
F.~Hausdorff.
\newblock Summationsmethoden und {M}omentfolgen. {I}.
\newblock {\em Math. Z.}, 9(1-2):74--109, 1921.

\bibitem{He15}
O.~H\'enard.
\newblock The fixation line in the {$\Lambda$}-coalescent.
\newblock {\em Ann. Appl. Probab.}, 25(5):3007--3032, 2015.

\bibitem{HeMo12}
P.~Herriger and M.~M\"ohle.
\newblock Conditions for exchangeable coalescents to come down from infinity.
\newblock {\em ALEA Lat. Am. J. Probab. Math. Stat.}, 9(2):637--665, 2012.

\bibitem{JK14}
S.~Jansen and N.~Kurt.
\newblock On the notion(s) of duality for markov processes.
\newblock {\em Probab. Surveys}, 11:59--120, 2014.

\bibitem{KKL01}
I.~Kaj, S.~M. Krone, and M.~Lascoux.
\newblock Coalescent theory for seed bank models.
\newblock {\em J. Appl. Probab.}, 38(2):285--300, 2001.

\bibitem{Ki62}
M.~Kimura.
\newblock {On the probability of fixation of mutant genes in a population.}
\newblock {\em Genetics}, 47:713--719, 1962.

\bibitem{Ki82}
J.~F.~C. Kingman.
\newblock The coalescent.
\newblock {\em Stochastic Process. Appl.}, 13(3):235--248, 1982.

\bibitem{KHB13}
S.~Kluth, T.~Hustedt, and E.~Baake.
\newblock The common ancestor process revisited.
\newblock {\em Bull. Math. Biol.}, 75(11):2003--2027, 2013.

\bibitem{KMTZ17}
B.~Koopmann, J.~M\"{u}ller, A.~Tellier, and D.~\v{Z}ivkovi\'{c}.
\newblock Fisher--{W}right model with deterministic seed bank and selection.
\newblock {\em Theor. Popul. Biol.}, 114:29--39, 2017.

\bibitem{KroNe97}
S.~M. Krone and C.~Neuhauser.
\newblock Ancestral processes with selection.
\newblock {\em Theor. Popul. Biol.}, 51(3):210--237, 1997.

\bibitem{K70}
T.~G. Kurtz.
\newblock Solutions of ordinary differential equations as limits of pure jump
  {M}arkov processes.
\newblock {\em J. Appl. Probab.}, 7:49--58, 1970.

\bibitem{leb}
N.~N. Lebedev.
\newblock {\em Special {F}unctions and their {A}pplications}.
\newblock Dover Publications, New York, 1972.
\newblock Revised edition, translated from the Russian and edited by Richard A.
  Silverman.

\bibitem{LKBW15}
U.~Lenz, S.~Kluth, E.~Baake, and A.~Wakolbinger.
\newblock Looking down in the ancestral selection graph: A probabilistic
  approach to the common ancestor type distribution.
\newblock {\em Theor. Popul. Biol.}, 103:27--37, 2015.

\bibitem{Li10}
T.~M. Liggett.
\newblock {\em Continuous {T}ime {M}arkov {P}rocesses. {A}n {I}ntroduction}.
\newblock American Mathematical Society, Providence, 2010.

\bibitem{MM00}
M.~M\"ohle.
\newblock Total variation distances and rates of convergence for ancestral
  coalescent processes in exchangeable population models.
\newblock {\em Adv. Appl. Probab.}, 32(4):983--993, 2000.

\bibitem{M14}
M.~M\"ohle.
\newblock Asymptotic hitting probabilities for the {B}olthausen--{S}znitman
  coalescent.
\newblock {\em J. Appl. Probab.}, 51(A):87--97, 2014.

\bibitem{M14b}
M.~M\"ohle.
\newblock On hitting probabilities of beta coalescents and absorption times of
  coalescents that come down from infinity.
\newblock {\em ALEA Lat. Am. J. Probab. Math. Stat.}, 11:141--159, 2014.

\bibitem{NeKro97}
C.~Neuhauser and S.~M. Krone.
\newblock The genealogy of samples in models with selection.
\newblock {\em Genetics}, 145(2):519--534, 1997.

\bibitem{Pit99}
J.~Pitman.
\newblock Coalescents with multiple collisions.
\newblock {\em Ann. Probab.}, 27(4):1870--1902, 1999.

\bibitem{Pit06}
J.~Pitman.
\newblock {\em Combinatorial {S}tochastic {P}rocesses}, volume 1875 of {\em
  Lecture Notes in Mathematics}.
\newblock Springer, Berlin, 2006.

\bibitem{PP13}
C.~Pokalyuk and P.~Pfaffelhuber.
\newblock The ancestral selection graph under strong directional selection.
\newblock {\em Theor. Popul. Biol.}, 87:25--33, 2013.

\bibitem{Sa99}
S.~Sagitov.
\newblock The general coalescent with asynchronous mergers of ancestral lines.
\newblock {\em J. Appl. Probab.}, 36(4):1116--1125, 1999.

\bibitem{Ta07}
J.~E. Taylor.
\newblock The common ancestor process for a {W}right--{F}isher diffusion.
\newblock {\em Electron. J. Probab.}, 12(28):808--847, 2007.

\bibitem{Tri}
F.~G. Tricomi.
\newblock {\em Integral Equations}.
\newblock Intercience Publishers, London, 1957.

\end{thebibliography}
\end{document}